\documentclass[12pt]{article}
\title{KIAS Lectures on\\ Symplectic Aspects of Degenerations}
\author{Jonny Evans}
\usepackage{amsthm,amsmath,amsfonts,amscd}
\usepackage{tikz}
\usepackage{verbatim}
\usepackage{caption}
\usepackage{pgfplots}
\usetikzlibrary{decorations.pathreplacing}
\usetikzlibrary{decorations.markings}
\usetikzlibrary{decorations.pathmorphing}
\tikzset{snake it/.style={decorate, decoration=snake}}
\tikzset{->-/.style={decoration={
              markings,
              mark=at position .5 with {\arrow{>}}},postaction={decorate}}}
\tikzset{->>-/.style={decoration={
              markings,
              mark=at position .5 with {\arrow{>>}}},postaction={decorate}}}
\tikzset{->>>-/.style={decoration={
              markings,
              mark=at position .5 with {\arrow{>>>}}},postaction={decorate}}}
\tikzset{->>>>-/.style={decoration={
              markings,
              mark=at position .5 with {\arrow{>>>>}}},postaction={decorate}}}
\newcommand{\conic}[1]{
\begin{scope}[shift={#1}]
\filldraw[fill=white,draw=none] (-0.5,-0.5) to[out=45,in=-45] (-0.5,0.5) -- (0.5,0.5) to[out=-135,in=135] (0.5,-0.5) -- (-0.5,-0.5);
\filldraw[fill=white] (0,0.5) circle [x radius=0.5,y radius=0.25];
\filldraw[fill=white,dashed] (-0.5,-0.5) arc [x radius=0.5,y radius=0.25,start angle=180,end angle=0];
\filldraw[fill=white] (-0.5,-0.5) arc [x radius=0.5,y radius=0.25,start angle=180,end angle=360];
\draw (-0.5,-0.5) to[out=45,in=-45] (-0.5,0.5);
\draw (0.5,-0.5) to[out=135,in=-135] (0.5,0.5);
\end{scope}
}
\newcommand{\nodalconic}[1]{
\begin{scope}[shift={#1}]
\filldraw[fill=white,draw=none] (-0.5,-0.5) -- (0,0) -- (0.5,-0.5) -- cycle;
\filldraw[fill=white,draw=none] (-0.5,0.5) -- (0,0) -- (0.5,0.5) -- cycle;
\filldraw[fill=white] (0,0.5) circle [x radius=0.5,y radius=0.25];
\filldraw[fill=white,dashed] (-0.5,-0.5) arc [x radius=0.5,y radius=0.25,start angle=180,end angle=0];
\filldraw[fill=white] (-0.5,-0.5) arc [x radius=0.5,y radius=0.25,start angle=180,end angle=360];
\draw (-0.5,-0.5) -- (0.5,0.5);
\draw (-0.5,0.5) -- (0.5,-0.5);
\end{scope}
}
\newcommand{\circdash}[4]{
\begin{scope}[shift={#1},scale=#2,rotate=#4]
\draw[#3,dashed] (-1,0) arc [x radius=1,y radius=0.5,start angle=180,end angle=0];
\draw[#3] (1,0) arc [x radius=1,y radius=0.5,start angle=0,end angle=-180];
\end{scope}
}
\definecolor{dg}{HTML}{228B22}
\usepackage{bm}
\usepackage{pinlabel}
\usepackage{parskip}
\usepackage{ragged2e}
\usepackage{float}
\usepackage{subfig}
\usepackage[utf8]{inputenc}
\newcommand{\CC}{\mathbf{C}}
\newcommand{\QQ}{\mathbf{Q}}
\newcommand{\RR}{\mathbf{R}}

\newcommand{\ZZ}{\mathbf{Z}}
\newcommand{\cp}[1]{\mathbf{CP}^{#1}}
\newcommand{\rp}[1]{\mathbf{RP}^{#1}}

\begingroup
    \makeatletter
    \@for\theoremstyle:=definition,remark,plain\do{%
        \expandafter\g@addto@macro\csname th@\theoremstyle\endcsname{%
            \addtolength\thm@preskip\parskip
            }%
        }
\endgroup
\usepackage{graphicx}
\usepackage[capitalise,noabbrev]{cleveref}
\newtheorem{Theorem}{Theorem}[section]
\newtheorem*{Theorem*}{Theorem}
\newtheorem{Lemma}[Theorem]{Lemma}

\newtheorem{Proposition}[Theorem]{Proposition}
\theoremstyle{remark}
\newtheorem{Remark}[Theorem]{Remark}
\theoremstyle{definition}
\newtheorem{Definition}[Theorem]{Definition}
\newtheorem{Example}[Theorem]{Example}
\newtheorem{Exercise}[Theorem]{Exercise}
\crefname{Theorem}{Theorem}{Theorems}
\Crefname{Theorem}{Theorem}{Theorems}
\crefname{Lemma}{Lemma}{Lemmas}
\Crefname{Lemma}{Lemma}{Lemmas}
\crefname{Corollary}{Corollary}{Corollaries}
\Crefname{Corollary}{Corollary}{Corollaries}
\crefname{Claim}{Claim}{Claims}
\Crefname{Claim}{Claim}{Claims}
\crefname{Proposition}{Proposition}{Propositions}
\Crefname{Proposition}{Proposition}{Propositions}
\crefname{Remark}{Remark}{Remarks}
\Crefname{Remark}{Remark}{Remarks}
\crefname{Definition}{Definition}{Definitions}
\Crefname{Definition}{Definition}{Definitions}
\crefname{Example}{Example}{Examples}
\Crefname{Example}{Example}{Examples}
\crefname{Exercise}{Exercise}{Exercises}
\Crefname{Exercise}{Exercise}{Exercises}
\begin{document}
\maketitle
This is a series of three lectures I gave at the Korea Institute
of Advanced Study in June 2019 at a workshop about ``Algebraic
and Symplectic Aspects of Degenerations of Complex Surfaces''. I
will focus on the symplectic aspects, in particular on the case
of cyclic quotient surface singularities.

I would like to thank the Korea Institute for Advanced Study for
their hospitality during this conference, and acknowledge the
support of the Engineering and Physical Sciences Research
Council (EPSRC Grant EP/P02095X/1 and 2) in carrying out the
research which forms the focus of my lectures.

These notes have been available on a public Git repository since
2019, and I noticed that people occasionally cited them in the
years since. For that reason, I decided to post them on arXiv
for a more permanent record; I have made some small corrections
and annotations but otherwise they are unchanged. These notes
are a purely expository account of stuff I was thinking about
2016--2019, and are largely self-aggrandising. I think people
found the early bits useful because they draw together some
things from the literature and put them in context. Note that
much has been superseded and explained better
elsewhere. Regarding Lecture 1, a simpler and more elegant
framework for thinking about degenerations has since been
developed by Galkin and Mikhalkin
\cite{GalkinMikhalkin}. Regarding Lecture 2 and parts of Lecture
3, a more detailed exposition is given in my book
\cite{EvansBook}.

\section{Lecture 1}

\subsection{Hunting for singularities}

Suppose that someone hands you a variety \(X\), and asks you to
classify degenerations in which \(X\) appears as a smooth fibre. This
is an extremely difficult question. In this first lecture, I want to
convince you that symplectic topology can help you to rule out some
possible degenerations. In the same way that a hunted animal leaves
footprints as it flees, a singularity sometimes leaves behind a
noticeable trace in the smooth fibres of a degeneration, a trace which
is visible to symplectic geometry.

\subsubsection{Degenerations and symplectic parallel transport}

What do I mean by a degeneration in this context? I mean a flat family
\(\pi\colon\mathcal{X}\to\Delta\) over the disc whose fibres
\(X_s=\pi^{-1}(s)\) are projective varieties. In fact, I'll assume for
simplicity:
\begin{itemize}
\item that the only singular fibre is \(X_0\), and that \(X_0\) has a
unique isolated singularity \(x_0\) (this is just for convenience),
\item that \(\mathcal{X}\) is smooth away from \(x_0\) (but will usually
be singular at \(x_0\)).
\end{itemize}
More importantly, I want to assume:
\begin{itemize}
\item that the fibres \(X_s\) are projective subvarieties of the same
projective space; in other words that we have a morphism
\(f\colon\mathcal{X}\to\cp{N}\) such that \(f_s:=f|_{X_s}\colon
X_s\to\cp{N}\) is an embedding for all \(s\in\Delta\) (for example,
coming from a relatively ample line bundle on \(\mathcal{X}\)).
\end{itemize}
It is this last point which allows us to make the connection with
symplectic geometry: the Fubini-Study form \(\omega_{FS}\) on
\(\cp{N}\) pulls back to give a symplectic form
\(\omega_s=f_s^*\omega_{FS}\) on the smooth locus of each fibre
\(X_s\).

\begin{Lemma}\label{lma:spt}
The smooth fibres are all {\em symplectomorphic}. More precisely,
given a path \(\gamma\colon[0,1]\to\Delta\) avoiding \(0\in\Delta\),
there is a diffeomorphism \(\phi_t\colon X_{\gamma(0)}\to
X_{\gamma(t)}\) for all \(t\in[0,1]\) such that
\(\phi_t^*\omega_{\gamma(t)}=\omega_{\gamma(0)}\).
\end{Lemma}
\begin{proof}
We construct \(\phi_t\) as the parallel transport map for a
connection on the fibre bundle \(\mathcal{X}\setminus
X_0\to\Delta\setminus\{0\}\). The connection is defined as
follows. Let \(\Omega=f^*\omega_{FS}\). This is a closed (not
necessarily nondegenerate\footnote{For example, suppose you have a
pencil of hypersurfaces in \(\cp{N}\). If you blow up the base
locus, this will separate out the fibres, and give you a
degeneration. However, the pullback of the Fubini-Study form along
the blow-up is degenerate along the exceptional locus.}) 2-form on
\(\mathcal{X}\setminus\{x_0\}\). Define the horizontal space at
\(x\in X_s\) to be \[\mathcal{H}_x=\{v\in T_x\mathcal{X}\ :\
\Omega(v,w)=0\ \forall w\in T_xX_s\}.\] This is the
\(\Omega\)-orthogonal complement to the vertical tangent space
\(T_xX_s\), which is indeed complementary to \(T_xX_s\) because
\(X_s\) is a symplectic submanifold with respect to \(\Omega\)
(\(\Omega\) pulls back to \(\omega_s\) on \(X_s\)). The proof that
parallel transport with respect to this connection is a symplectic
map can be found in {\cite[Lemma 6.18]{McDuffSalamon}}. \qedhere

\end{proof}
Now suppose that \(\gamma(1)=0\) and \(\gamma(t)\neq 0\) for
\(t<1\). The symplectic parallel transport maps \(\phi_t\colon
X_{\gamma(0)}\to X_{\gamma(t)}\) are defined for \(t<1\). We would
like to define \(\phi_1\colon X_{\gamma(0)}\to X_0\) by
\[\phi_1(x)=\lim_{t\to 1}\phi_t(x).\]

\begin{Lemma}
This map \(\phi_1\) is well-defined.
\end{Lemma}
\begin{proof}
Since \(\mathcal{X}\) is compact, any sequence \(\phi_{t_i}(x)\),
\(t_i\in[0,1]\) has a convergent subsequence and if \(t_i\to 1\)
then the limit is in the fibre \(X_0\). Pick a convergent
subsequence and define \(\phi_1(x)\) to be the limit. We need to
show that this limit is independent of the choice of convergent
subsequence.

Suppose the limit \(y=\lim\phi_{t_i}(x)\) of some subsequence is a
smooth point of \(X_0\). In this case, the horizontal space is still
well-defined at \(y\), and if we look in a neighbourhood of \(y\)
then the parallel transport problem (an ordinary differential
equation) is well-posed. In particular, the only way this can happen
is if \(\phi_t(x)\) for \(t\in[1-\epsilon,1]\) is a solution of the
parallel transport ODE on this neighbourhood. In that case, by
uniqueness of solutions to ODEs, any other subsequence will converge
to the same limit point \(y\).

The only other possibility is that no subsequence converges to a
smooth point, in which case every subsequence converges to \(x_0\)
(because there's only one singularity by assumption). \qedhere

\end{proof}
In more general contexts, for example with several singularities or
even nonisolated singularities, you could use prove the existence of
this map \(\phi_1\) using more sophisticated arguments involving
\L{}ojasiewicz's inequality (for example, this approach is taken by
Harada and Kaveh in their work on toric degenerations).

\subsubsection{Signs left by singularities}

We now use the map \(\phi_1\colon X_{\gamma(0)}\to X_0\) to find
objects in the smooth fibre \(X_{\gamma(0)}\) which are signs that the
singularity is present in \(X_0\). The three signs we look for are:
the {\em vanishing cycle}, the {\em link} and the {\em Milnor fibre}.

\begin{Definition}[Vanishing cycle]\label{dfn:vc}
Let \(V:=\{x\in X_{\gamma(0)}\ :\ \phi_1(x)=x_0\}\). This is called
the {\em vanishing cycle} of the singularity (with respect to the
path \(\gamma\)).

\end{Definition}
\begin{Remark}
There is no {\em a priori} reason for the vanishing cycle to be
nice, e.g. a submanifold or cell complex. It will often turn out in
practice to be a Lagrangian cell complex, that is a cell complex of
dimension \(\frac{1}{2}\dim X\) on which the symplectic form
vanishes.

\end{Remark}
\begin{Definition}[Link]\label{dfn:link}
Let \(S_\epsilon\subset\cp{N}\) be the sphere of radius \(\epsilon\)
centred at \(f(x_0)\) and let
\(L_\epsilon=f_0^{-1}(S^{2N+1}_\epsilon)\cap X_0\). For sufficiently
small \(\epsilon\), \(L_\epsilon\) is a contact-type hypersurface
called the {\em link} of \(x_0\).

\end{Definition}
\begin{Remark}
Recall that a hypersurface \(M\) of a \(2n\)-dimensional symplectic
manifold is said to be of {\em contact type} if there exists a
1-form \(\lambda\) on a neighbourhood of \(M\) such that
\(\omega=d\lambda\) on this neighbourhood and such that
\(\lambda\wedge \underbrace{d\lambda \wedge \cdots\wedge
d\lambda}_{n-1}\) is a nowhere vanishing \((2n-1)\)-form on \(M\)
(the pullback of \(\lambda\) to \(M\) is called a {\em contact
form}).

\end{Remark}
\begin{Lemma}
There is a contact-type hypersurface
\(M:=\phi_1^{-1}(L_{\epsilon})\subset X_{\gamma(0)}\) which is
contactomorphic to the link of \(x_0\).
\end{Lemma}
\begin{proof}
The map \(\phi_1\) restricts to give a map \(X_{\gamma(0)}\setminus
V\to X_0\setminus\{x_0\}\). While \(\phi_1\) is only continuous,
this restriction is smooth. This is because it is defined as the
time 1 map of an ordinary differential equation (the parallel
transport problem) which is well-posed wherever the horizontal
spaces are well-defined (in particular anywhere except \(x_0\)). The
argument from {\cite[Lemma 6.18]{McDuffSalamon}} carries through and
shows that it is symplectic. Therefore, as \(L_{\epsilon}\) is
disjoint from \(x_0\), \(M:=\phi_1^{-1}(L_\epsilon)\) is still a
contact-type hypersurface in \(X_{\gamma(0)}\) and \((\phi_1)|_{M}\)
is the desired contactomorphism to \(L_\epsilon\). \qedhere

\end{proof}
\begin{Definition}[Milnor fibre]\label{dfn:milnorfibre}
The contact-type hypersurface \(M\) divides \(X_{\gamma(0)}\) into
two regions: a region \(W\) containing the vanishing cycle \(V\) and
its complement. The region \(W\) is called the {\em Milnor fibre} of
the singularity. It is a codimension zero symplectic submanifold
which is mapped via \(\phi_1\) onto a neighbourhood of \(x_0\).

\end{Definition}
\begin{Remark}
We will often be slightly vague about Milnor fibres. If someone
tells you they have found a symplectically embedded ball in a
manifold, you should ask them: how big a ball? Balls of different
radii are not symplectomorphic (they have different volumes) and
there are extremely subtle questions about how big a ball you can
embed in a given manifold. Similarly, just because you find two
singularities which are locally analytically equivalent, the Milnor
fibres you find may not be symplectomorphic: what will be true
instead is that you can find inside each of them a common smaller
subset. Said another way, the Milnor fibres of two analytically
equivalent isolated singularities will have symplectomorphic {\em
completions} (if you attach an infinite cylindrical end); see
{\cite[Proposition 11.22]{CieliebakEliashberg}}. Apart from in
\cref{cor:mori}, we will ignore these subtle quantitative questions
as they have little bearing on what we want to say, and always allow
ourselves to pass to a smaller neighbourhood of the vanishing cycle
if necessary, using the same notation for a Milnor fibre and its
completion.

\end{Remark}
\subsection{Cyclic quotient singularities}

We now take a concrete example, and identify the link, the Milnor
fibre and the vanishing cycle. Let \(G\) be the cyclic group of
\(n\)th roots of unity and let \(\mu\in G\) act on \(\CC^2\)
(coordinates \((x,y)\)) via \[(x,y)\mapsto (\mu x,\mu^a y),\] where
\(a\) is an integer coprime to \(n\). The quotient \(\CC^2/G\) by this
group action has an isolated singularity at the origin, which we call
the {\em cyclic quotient singularity of type} \(\frac{1}{n}(1,a)\). We
will say that \(X_0\) has a singularity \(x_0\) of this type if
\(x_0\) has a neighbourhood (in the Euclidean topology) which is
biholomorphic to a neighbourhood of \(0\in\CC^2/G\).

\subsubsection{The link}

The unit sphere in \(\CC^2\) is preserved by the action of \(G\) and
the quotient of the sphere by this action is a 3-manifold called a
{\em lens space} \(L(n,a)\). This means that the link of the
\(\frac{1}{n}(1,a)\) singularity is such a lens space. Therefore if
you have a smooth variety \(X\) which does not contain any
contact-type hypersurfaces contactomorphic\footnote{We equip
\(L(n,a)\) with the contact structure it inherits as a quotient of the
standard contact structure on \(S^3\).} to \(L(n,a)\) then it does not
admit any degenerations with \(\frac{1}{n}(1,a)\) singularities.

\subsubsection{Milnor fibre}

Cyclic quotient singularities can have different smoothings, so there
can be several different possible Milnor fibres to look out for. The
smoothings were classified by Koll\'{a}r and Shepherd-Barron
\cite{KSB} and we will discuss this in \cref{sct:ksb}. In the
meantime, we will focus on a specific, important class of examples.

\begin{Definition}[T-singularities]\label{dfn:tsing}
A cyclic quotient singularity is called a T-singularity if it is of
the form \(\frac{1}{dp^2}(1,dpq-1)\) with \(\gcd(p,q)=1\).

\end{Definition}
\begin{Remark}\label{rmk:tsing}
These are precisely the cyclic quotient singularities admitting a
\(\QQ\)-Gorenstein smoothing with terminal total space, hence the
name ``T''.

\end{Remark}
Fix a T-singularity \(\frac{1}{dp^2}(1,dpq-1)\). If we define the
semi-invariants \[u=x^{dp},\quad v=y^{dp},\quad w=xy\] then:
\begin{itemize}
\item \(uv=w^p\)
\item under the action of \(\mu\in G\),
\[(u,v,w)\mapsto(\mu^{dp}u,\mu^{-dp}v,\mu^{dpq}w)\]
(e.g. \(y^{dp}\mapsto(\mu^{dpq-1}y)^{dp}=\mu^{-dp}y^{dp}\) because
\(\mu^{d^2p^2q}=1\)).
\end{itemize}
In fact, we get \[\CC[x,y]^G=\CC[u,v,w]^{G'}/(uv=w^{dp}),\] where
\(G'\) is the cyclic group of order \(p\) generated by
\(\mu^{dp}\). The family of varieties \[X_s:=\{(u,v,w)\in\CC^3\ :\
uv=(w^p-s)(w^p-2s)\cdots(w^p-ds)\}/G'\] is therefore a smoothing of
\(\frac{1}{dp^2}(1,dpq-1)\) (here \(G'\) acts with weights \(1,-1,q\)
on \(u,v,w\) respectively, and \(s\) is the smoothing parameter).

\begin{Remark}
Note that the action of \(G'\) on
\(uv=(w^p-s)(w^p-2s)\cdots(w^p-ds)\) is free, so \(X_s\) is a smooth
variety.

\end{Remark}
\begin{Definition}
We define the manifold \(B_{d,p,q}\) to be the symplectic manifold
underlying the affine variety \(X_1\) (we could take any \(X_s\),
\(s\neq 0\), as they are all symplectomorphic by parallel
transport). This is (the symplectic completion of) the Milnor fibre
of the cyclic quotient singularity \(\frac{1}{dp^2}(1,dpq-1)\).

\end{Definition}
\begin{Remark}
Momentarily, we will study the vanishing cycle and show that it is a
Lagrangian cell complex (i.e. a 2-dimensional cell complex on which
the symplectic form vanishes) which comprises a chain of \(d-1\)
Lagrangian spheres attached to a certain immersed Lagrangian disc
called a {\em pinwheel}. Indeed, \(B_{d,p,q}\) deformation retracts
onto the vanishing cycle, which tells us that
\[H_1(B_{d,p,q};\ZZ)=\ZZ/p,\qquad H_2(B_{d,p,q};\ZZ)=\ZZ^{d-1}.\] In
the applications below, we are most interested in the case \(d=1\)
as the Milnor fibre is then a rational homology ball, so it can be
difficult to rule out its appearance on purely topological
grounds. In this case, we write \(B_{p,q}:=B_{1,p,q}\).

\end{Remark}
\subsubsection{Vanishing cycle}

\begin{Definition}
Consider the map \(S^1\to S^1\), \(z\mapsto z^p\). The cell complex
obtained by attaching a 2-cell to \(S^1\) using this as the
attaching map is called a {\em pinwheel}. Here is picture indicating
how to identify segments of the boundary of the disc to get a
pinwheel (for example, when \(p=2\), this is the usual picture of
\(\rp{2}\) as a disc with opposite pairs of points on the boundary
identified).

\end{Definition}
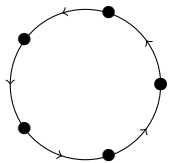
\begin{center}
\tikzset{circarrows/.style={postaction={decorate},decoration={markings,mark=at position 0.1 with {\arrow{>}}},
decoration={markings,mark=at position 0.3 with {\arrow{>}}},
decoration={markings,mark=at position 0.5 with {\arrow{>}}},
decoration={markings,mark=at position 0.7 with {\arrow{>}}},
decoration={markings,mark=at position 0.9 with {\arrow{>}}},}}
\begin{tikzpicture}
\draw[circarrows] (0,0) circle [x radius=1,y radius=1];
\node at (0:1) {\(\bullet\)};
\node at (72:1) {\(\bullet\)};
\node at (144:1) {\(\bullet\)};
\node at (216:1) {\(\bullet\)};
\node at (288:1) {\(\bullet\)};
\node at (360:1) {\(\bullet\)};

\end{tikzpicture}
\end{center}
\begin{Proposition}
Let \(\gamma(t)=1-t\). The vanishing cycle associated to \(\gamma\)
in \(B_{d,p,q}=X_1\) is a cell complex comprising a chain of \(d-1\)
Lagrangian spheres attached to a Lagrangian pinwheel.

\end{Proposition}
Here, {\em Lagrangian} means that the symplectic form \(\omega_1\)
evaluates to zero on pairs of tangent vectors to 2-cells in the
vanishing cycle. The rest of this section will be spent justifying
this claim.

Consider the projection
\begin{equation}\label{eq:conicfib}B_{d,p,q}\to\CC,\qquad
(u,v,w)\mapsto w^p.\end{equation} This is a holomorphic map whose
general fibre is a conic \(uv=\mbox{constant}\). It has the following
singular fibres:
\begin{itemize}
\item over \(w^p=1,2,3,\ldots,d\) there are nodal conics \(uv=0\);
\item over \(w^p=0\) the fibre is a smooth conic but it is not reduced: it
is the quotient by \(G'\) of \(\{uv=d!\}\). The generic fibre
collapses \(p\)-to-\(1\) onto this fibre.
\end{itemize}
In the figure below we indicate these singular fibres.

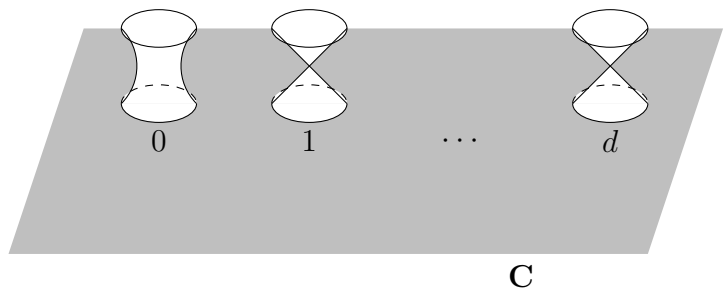
\begin{center}
\begin{tikzpicture}
\filldraw[fill=lightgray,opacity=0.5,draw=none] (0,-0.5) -- (8.5,-0.5) -- (9.5,2.5) -- (1,2.5) -- cycle;
\node at (6.5,-0.5) [below right] {\(\CC\)};
\conic{(2,2)};
\nodalconic{(4,2)};
\nodalconic{(8,2)};
\node at (2,1) {\(0\)};
\node at (4,1) {\(1\)};
\node at (6,1) {\(\cdots\)};
\node at (8,1) {\(d\)};

\end{tikzpicture}
\end{center}
The following lemma gives a way to construct Lagrangian submanifolds
in the total space of a conic fibration (or more general degeneration)
from Lagrangians in the fibres. We also state and prove the (trickier)
converse because we will use it later.

\begin{Lemma}\label{lma:spt-lag}
Suppose \(\pi\colon\mathcal{Y}\to\CC\) is a degeneration/conic
fibration and that \(\Omega\) is a closed 2-form on \(\mathcal{Y}\)
whose pullback to fibres of \(\pi\) is nondegenerate. Let
\(\gamma\colon[0,1]\to\CC\) be a path in the base of the fibration
and let \(\phi_t\colon Y_{\gamma(0)}\to Y_{\gamma(t)}\) be the
symplectic parallel transport map. If \(L\subset Y_{\gamma(0)}\) is
a Lagrangian in the fibre then \(\bigcup_{t\in[0,1]}\phi_t(L)\) is
\(\Omega\)-Lagrangian \(\mathcal{Y}\). Conversely, if
\(\mathcal{L}\subset\mathcal{Y}\) is a Lagrangian which lives
submersively over \(\gamma\) then
\(\mathcal{L}=\bigcup_{t\in[0,1]}\phi_t(\mathcal{L}\cap
Y_{\gamma(0)})\).
\end{Lemma}
\begin{proof}
If \(L\subset Y_{\gamma(0)}\) is Lagrangian and \(\xi\) is a
horizontal vector at \(x\) then \(\omega(T_xL,\xi)=0\) because
\(\xi\) is symplectically orthogonal to fibres. If \(\xi\) is the
horizontal lift of \(\dot{\gamma}\) then the submanifold
\(\bigcup_{t\in[0,1]}\phi_t(L)\) traced out by \(L\) under parallel
transport has tangent space \(T_xL\oplus\langle\xi\rangle\), which
is therefore Lagrangian.

Conversely, suppose \(\mathcal{L}\subset\mathcal{Y}\) is a
Lagrangian living submersively over \(\gamma\). Let
\(x\in\mathcal{L}\) be a point with \(\pi(x)=\gamma(t)=:s\) and let
\(\xi\) be a tangent vector to \(\mathcal{L}\) whose projection
\(\pi_*\xi\) is \(\dot{\gamma}(t)\). Let \(\xi'\in\mathcal{H}_x\) be
the horizontal lift of \(\dot{\gamma}(t)\). Since
\(\pi_*\xi=\pi_*\xi'\), we have \(\xi=\xi'+v\) for some vertical
vector \(v\). If we pick a symplectic
basis\footnote{i.e. \(\Omega(e_i,f_j)=\delta_{ij}\).}
\(e_1,\ldots,e_{\dim Y_s},f_1,\ldots,f_{\dim Y_s}\) for \(T_xY_s\)
with \(e_1,\ldots,e_{\dim Y_s}\in T_x\mathcal{L}\) then there are
numbers \(\alpha_i,\beta_i\) such that
\[\xi=\xi'+\sum\alpha_ie_i+\sum\beta_if_i.\] Since \(\mathcal{L}\)
is Lagrangian, \(\Omega(\xi,e_j)=0\) for all \(j\). Since \(\xi'\)
is horizontal, \(\Omega(\xi',e_j)=0\) for all \(j\). Therefore
\(\beta_j=0\) for all \(j\), which means
\(\xi'=\xi-\sum\alpha_ie_i\in T_x\mathcal{L}\). That is, \(\xi'\) is
tangent to \(\mathcal{L}\), which means that \(\mathcal{L}\) is
preserved by parallel transport. \qedhere

\end{proof}
In our case, this means we can construct Lagrangians in \(B_{d,p,q}\)
just by taking a loop in a conic fibre and transporting it over a path
in \(\CC\). The difficulty is in actually solving the parallel
transport equation and seeing where a loop goes under parallel
transport. Happily, there is a useful conserved quantity which solves
this problem for us.

\begin{Lemma}\label{lma:sptconic}
The function \(H:=\frac{1}{2}(|u|^2-|v|^2)\colon B_{d,p,q}\to\RR\)
is preserved by symplectic parallel transport in the conic fibration
\cref{eq:conicfib}.
\end{Lemma}
\begin{proof}
Recall that any function \(H\) on a symplectic manifold defines a
{\em Hamiltonian vector field} \(V_H\) satisfying
\(\iota_{V_H}\omega=-dH\). The simplest example of this is if
\(\omega=dp\wedge dq\) on \(\RR^2\), in which case \(V_H=(-\partial
H/\partial q,\partial H/\partial p)\). In particular, if
\(H=\frac{1}{2}(p^2-q^2)\) then \(V_H=(-q,p)\), which generates a
flow by rotations of \(\RR^2\). Therefore, the Hamiltonian
\(\frac{1}{2}(|u|^2-|v|^2)\) generates the flow
\((e^{i\theta}u,e^{-i\theta}v,w)\) on \(B_{d,p,q}\). This flow
preserves the fibres \(uv=\mbox{const}\) so \(V_H\) is vertical. If
\(\xi\) is a horizontal vector field then
\[\mathcal{L}_{\xi}H=dH(\xi)\] but
\(dH(\xi)=-\omega(V_H,\xi)\). Since horizontal vectors are, by
definition, symplectically orthogonal to vertical vectors, this
means \(dH(\xi)=0\), so \(H\) is preserved by parallel
transport. \qedhere

\end{proof}

For each conic fibre \(\pi^{-1}(s)\), let \(C_s\) be the circle
\(H^{-1}(0)\cap \pi^{-1}(s)\). \Cref{lma:spt-lag} tells us that
\(V:=\bigcap_{t\in[0,1]}C_{\gamma(t)}\) is a Lagrangian submanifold of
\(B_{d,p,q}\).

\begin{Example}
If we look at the path \(\gamma_k(t)=k+t\) for
\(k=0,1,2,\ldots,d-1\) then:
\begin{itemize}
\item for \(k=1,2,\ldots,d-1\) we get a Lagrangian sphere (the circles
\(C_{\gamma(t)}\) collapse to points as \(t\to 0\) and \(t\to
1\)).
\item for \(k=0\) we get a Lagrangian pinwheel. This is because the
fibre over \(0\) is nonreduced, so the symplectic parallel
transport map collapses \(C_t\) \(p\)-to-\(1\) onto \(C_0\).

\end{itemize}
\end{Example}
In the picture below, we draw this configuration \(V\) in the case
\(d=3\):

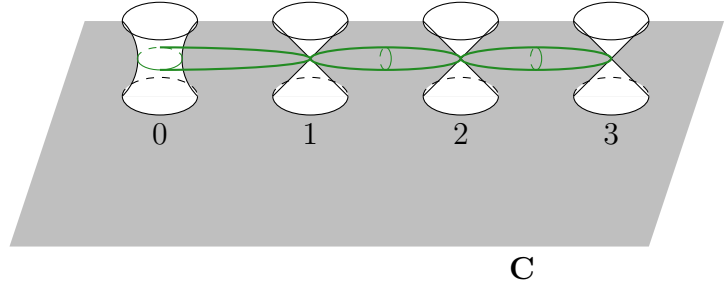
\begin{center}
\begin{tikzpicture}
\filldraw[fill=lightgray,opacity=0.5,draw=none] (0,-0.5) -- (8.5,-0.5) -- (9.5,2.5) -- (1,2.5) -- cycle;
\node at (6.5,-0.5) [below right] {\(\CC\)};
\conic{(2,2)};
\nodalconic{(4,2)};
\nodalconic{(8,2)};
\nodalconic{(6,2)}
\node at (2,1) {\(0\)};
\node at (4,1) {\(1\)};
\node at (6,1) {\(2\)};
\node at (8,1) {\(3\)};
\circdash{(2,2)}{0.3}{dg}{0};
\draw[thick,dg] (2,2.15) arc [x radius=2,y radius=0.15,start angle=90, end angle=-90];
\circdash{(5,2)}{0.15}{dg}{90};
\draw[thick,dg] (5,2) circle [x radius=1,y radius=0.15];
\circdash{(7,2)}{0.15}{dg}{90};
\draw[thick,dg] (7,2) circle [x radius=1,y radius=0.15];

\end{tikzpicture}
\end{center}
Finally, we show that this particular configuration \(V\) is the
vanishing cycle for the degeneration of \(B_{d,p,q}\) to the
T-singularity \(\frac{1}{dp^2}(1,dpq-1)\).

\begin{Lemma}
If we take the path \(\gamma(t)=1-t\) in the base of the family
\(X_s\) then the vanishing cycle in \(X_1=B_{d,p,q}\) associated to
the quotient singularity of \(X_0\) is precisely \(V\).
\end{Lemma}
\begin{proof}
Equip \(\{(u,v,w,s)\in\CC^4\ :\ uv=(w^p-s)\cdots(w^p-ds)\}\) with
the standard symplectic form coming from \(\CC^4\). Since the group
\(G'\) acts by symplectomorphisms on \(\CC^4\), this form descends
to a symplectic form \(\Omega\) on the total space
\(\mathcal{X}=\bigcup_{s\in\CC}X_s\) of our degeneration.

By \cref{lma:spt-lag}, we need to construct a Lagrangian submanifold
\(\mathcal{L}\) of \((\mathcal{X},\Omega)\) such that
\(\mathcal{L}\cap X_1=V\) and such that \(\mathcal{L}\cap X_0\)
contains \(x_0\) (using the word submanifold in the loosest sense at
this point).

Consider the antisymplectic involutions
\[\sigma_{\pm}\colon\CC^4\to\CC^4,\qquad\sigma_{\pm}(u,v,w,s)=(\pm\bar{v},\pm\bar{u},\bar{w},\bar{s}).\]
These preserve the equation \(uv=(w^p-s)\cdots(w^p-ds)\), and
commute with the action of \(G'\), so descend to give antisymplectic
involutions (which we continue to denote by \(\sigma_{\pm}\) of
\(\mathcal{X}\).

The fixed locus of an antisymplectic involution is a Lagrangian
submanifold, and
\(\mathcal{L}:=\mathrm{Fix}(\sigma_{-})\cup\mathrm{Fix}(\sigma_+)\)
intersects \(X_1\) in \(\bigcup_{t\in\RR}C_t\) (recall that
\(C_{\star}\) is the circle \(H^{-1}(0)\cap\pi^{-1}(\star)\) where
\(\pi\) is the conic fibration \((u,v,w)\mapsto w^p\)). From this it
is easy to see that \(V=\bigcup_{t\in[0,d]}C_t\subset X_1\) is the
vanishing cycle, see the figure below. \qedhere

\end{proof}
In this figure, we show a ``movie'' as \(s\) varies over \([0,1]\) of
the Lagrangians \(\mathrm{Fix}(\sigma_-)\cap X_s\) (red) and
\(\mathrm{Fix}(\sigma_+)\cap X_s\) (green). In this case, \(d=3\). If
\(d\) were even then both of the noncompact pieces of the Lagrangian
would be red.

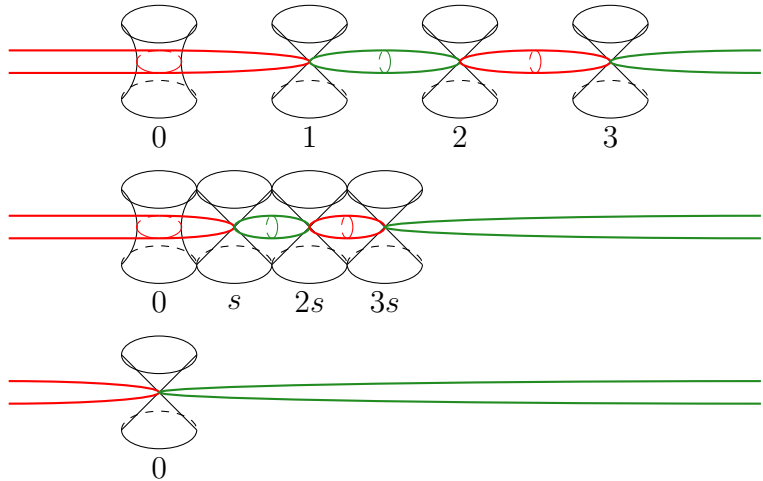
\begin{center}
\begin{tikzpicture}
\conic{(2,2)};
\nodalconic{(4,2)};
\nodalconic{(8,2)};
\nodalconic{(6,2)}
\node at (2,1) {\(0\)};
\node at (4,1) {\(1\)};
\node at (6,1) {\(2\)};
\node at (8,1) {\(3\)};
\circdash{(2,2)}{0.3}{red}{0};
\draw[thick,red] (2,2.15) -- (0,2.15);
\draw[thick,red] (2,1.85) -- (0,1.85);
\draw[thick,red] (2,2.15) arc [x radius=2,y radius=0.15,start angle=90, end angle=-90];
\circdash{(5,2)}{0.15}{dg}{90};
\draw[thick,dg] (5,2) circle [x radius=1,y radius=0.15];
\circdash{(7,2)}{0.15}{red}{90};
\draw[thick,red] (7,2) circle [x radius=1,y radius=0.15];
\draw[thick,dg] (10,2.15) arc [x radius=2,y radius=0.15,start angle=90,end angle=270];
\begin{scope}[shift={(0,-2.2)}]
\conic{(2,2)};
\nodalconic{(3,2)};
\nodalconic{(4,2)};
\nodalconic{(5,2)}
\node at (2,1) {\(0\)};
\node at (3,1) {\(s\)};
\node at (4,1) {\(2s\)};
\node at (5,1) {\(3s\)};
\circdash{(2,2)}{0.3}{red}{0};
\draw[thick,red] (2,2.15) -- (0,2.15);
\draw[thick,red] (2,1.85) -- (0,1.85);
\draw[thick,red] (2,2.15) arc [x radius=1,y radius=0.15,start angle=90, end angle=-90];
\circdash{(3.5,2)}{0.15}{dg}{90};
\draw[thick,dg] (3.5,2) circle [x radius=0.5,y radius=0.15];
\circdash{(4.5,2)}{0.15}{red}{90};
\draw[thick,red] (4.5,2) circle [x radius=0.5,y radius=0.15];
\draw[thick,dg] (10,2.15) arc [x radius=5,y radius=0.15,start angle=90,end angle=270];
\end{scope}
\begin{scope}[shift={(0,-4.4)}]
\nodalconic{(2,2)};
\node at (2,1) {\(0\)};
\draw[thick,red] (0,2.15) arc [x radius=2,y radius=0.15,start angle=90,end angle=-90];
\draw[thick,dg] (10,2.15) arc [x radius=8,y radius=0.15,start angle=90,end angle=270];
\end{scope}

\end{tikzpicture}
\end{center}
\newpage

\section{Lecture 2}

In the previous lecture, we saw that the symplectic geometry of a
smooth variety contains clues about the kinds of singularities that
can form as the variety degenerates. We worked out in detail the
example of cyclic quotient T-singularity \(\frac{1}{dp^2}(1,dpq-1)\),
identifying the link (a contact-type lens space \(L(dp^2,dpq-1)\)),
the Milnor fibre (a codimension zero symplectic submanifold
\(B_{d,p,q}\)) and the vanishing cycle (a chain of \(d-1\) Lagrangian
spheres attached to a Lagrangian pinwheel).

In this lecture, I want to explore the symplectic geometry of these
manifolds \(B_{d,p,q}\) more deeply and introduce a different way of
representing them: the ``almost toric pictures'' discovered by
Symington \cite{Symington}. By the end of the lecture, we will be able
to draw similar pictures of any smoothing of a cyclic quotient
singularity. These almost toric pictures will be a key ingredient in
Lecture 3.

\subsection{Toric and almost toric pictures}

To construct the family of \(B_{p,q}\)s, we will introduce a new way
of representing \(B_{d,p,q}\) called an {\em almost toric picture}. In
this new representation, symplectically embedded \(B_{d,p,q}\)s are
very visible. Almost toric pictures are generalisations of toric
moment images, which we review first.

\subsubsection{Toric varieties}

Recall from \cref{lma:sptconic} that a function \(H\) on a symplectic
manifold gives rise to a Hamiltonian vector field \(V_H\) such that
\(\iota_{V_H}\omega=-dH\) and a Hamiltonian flow (the flow of
\(V_H\)). If this flow is periodic with period \(2\pi\) then we call
the resulting circle action a {\em Hamiltonian circle action}.

If we have two Hamiltonian vector fields, associated to functions
\(H_1,H_2\) then the Lie bracket \([V_{H_1},V_{H_2}]\) is again
Hamiltonian, generated by the function
\(\{H_1,H_2\}=\omega(V_{H_1},V_{H_2})\). This function is called {\em
the Poisson bracket} of \(H_1,H_2\). In particular, if the Poisson
bracket vanishes then the flows commute (the converse is not true,
e.g. consider the flows generated by the \(x\) and \(y\) coordinates
on \(\RR^2\)).

\begin{Definition}
Suppose we have a collection of Hamiltonians \(H_1,\ldots,H_n\) such
that the Poisson brackets \(\{H_i,H_j\}\) vanish and such that the
lattice of periods \[\Lambda=\{(t_1,\ldots,t_n)\in\RR^n\ :\
\phi_{H_1}^{t_1}\cdots\phi_{H_n}^{t_n}=id\}\] is
\((2\pi\ZZ)^n\). Then we get an action of the torus
\(\RR^n/\Lambda\) which we call a {\em Hamiltonian torus action}.

\end{Definition}
\begin{Remark}
If the Hamiltonian vector fields are linearly independent at some
point \(x\) then \(n\leq \dim X/2\). To see this, note that the
orbit through \(x\) has tangent space spanned by the Hamiltonian
vector fields, and \(\omega\) vanishes on these vectors because the
Poisson brackets are zero. Therefore the orbit is {\em isotropic}
(the symplectic form vanishes on it) and an isotropic subspace can
have dimension at most \(\dim X/2\); if it has dimension \(\dim
X/2\) then it is called {\em Lagrangian}. If \(n=\dim X/2\) then we
say that \(X\) is {\em toric}.

\end{Remark}
\begin{Example}\label{exm:c2}
Consider the Hamiltonians \(H_1=\frac{1}{2}|x|^2\),
\(H_2=\frac{1}{2}|y|^2\) on \(\CC^2\). These Hamiltonians generate
the Hamiltonian torus action
\((x,y)\mapsto(e^{it_1}x,e^{it_2}y)\). The image of \(\CC^2\) by the
map \(\mu=(H_1,H_2)\colon\CC^2\to\RR^2\) is the nonnegative quadrant
in \(\RR^2\). The fibres of the map \(\mu\) over the interior of
this quadrant are Lagrangian tori. The fibres over the edges are
circles and the fibre over the origin is a single point.

\end{Example}
\begin{Definition}
The map \(\mu\colon X\to\RR^n\) given by \(\mu=(H_1,\ldots,H_n)\) is
called the {\em moment map} and its regular fibres are Lagrangian
tori. The image \(\mu(X)\subset\RR^n\) is a (possibly noncompact)
convex polytope called the {\em moment image}. It is actually
possible to reconstruct \(X\) (up to equivariant symplectomorphism)
from \(\mu(X)\); this is Delzant's theorem {\cite[Theorem
2.1]{Delzant}}.

\end{Definition}
\begin{Example}\label{exm:quotsing}
The Hamiltonians \(H_1,H_2\) from \cref{exm:c2} descend to the
quotient \(\CC^2/G\) where \(G\) is the group of \(n\)th roots of
unity acting by \((x,y)\mapsto (\mu x,\mu^a y)\). The quotient
singularity \(\frac{1}{n}(1,a)\) is therefore toric, however if we
simply use \(H_1,H_2\) as before then the period lattice is not
standard: we have \(\phi_{H_1}^{2\pi/n}\phi_{H_2}^{2\pi
a/n}=id\). If instead we use the Hamiltonians
\[\left(H_2,\frac{1}{n}(H_1+aH_2)\right)\]
then the lattice of periods becomes standard. The moment image is
a convex wedge we will denote by \(\pi(n,a)\):

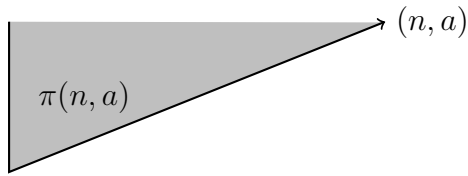
\begin{center}
\begin{tikzpicture}
\filldraw[draw=black,thick,->,fill=lightgray] (0,2) -- (0,0) -- (5,2);
\node at (1,1) {\(\pi(n,a)\)};
\node at (5,2) [right] {\((n,a)\)};

\end{tikzpicture}
\end{center}
\end{Example}
\begin{Remark}
In this example, we made a \(\QQ\)-affine change of Hamiltonians to
change the period lattice. Note that \(\ZZ\)-affine changes of
Hamiltonians leave the period lattice unchanged. In fact, if two
toric manifolds have moment images which are related by a
\(\ZZ\)-affine change of coordinates then they are equivariantly
symplectomorphic.

\end{Remark}
\begin{Remark}
In this example, we can see the contact lens space which is the link
of the \(\frac{1}{n}(1,a)\)-singularity. It is simply the preimage
of a horizontal line segment \(\sigma\) running across
\(\pi(n,a)\). In 3-dimensional topology, lens spaces are defined to
be those orientable 3-manifolds containing a torus whose complement
comprises two solid tori. In our toric picture, the preimage of the
midpoint of \(\sigma\) is a Lagrangian torus. The preimages of the
left- and right-hand segments in \(\sigma\) are solid tori, so we
see that this is indeed a lens space (you have to work a little
harder to see that it is \(L(n,a)\)).

\end{Remark}
\subsubsection{Lagrangian torus fibrations}

\begin{Definition}
Let \(X\) be a \(2n\)-dimensional symplectic manifold and \(B\) be
an \(n\)-dimensional stratified space. A Lagrangian torus fibration
is a proper continuous map \(F\colon X\to B\) such that:
\begin{itemize}
\item \(F\) is a smooth submersion over the top stratum of \(B\), with
Lagrangian fibres;
\item the fibres over other points are themselves stratified spaces,
with isotropic strata.

\end{itemize}
\end{Definition}
\begin{Remark}
The Arnold-Liouville theorem implies that the fibres over the top
stratum are tori.

\end{Remark}
\begin{Example}
If \(X\) is toric then we can take \(B\) to be the moment image and
\(F\) to be the moment map; this gives a Lagrangian torus
fibration. The moment image is a convex polytope, so it is
stratified by its faces. The top stratum is the interior: the fibres
over this stratum are Lagrangian tori. Over points in faces of
dimension \(k\), the fibres are isotropic tori of dimension \(k\).

\end{Example}
\begin{Example}
Consider \[B_{d,p,q}=\{(u,v,w)\in\CC^3\ :\
uv=(w^p-1)\cdots(w^p-d)\}/G'\] and the holomorphic map
\(B_{d,p,q}\to\CC\) given by \((u,v,w)\mapsto w^p\). Consider the
functions \(H_1=\frac{1}{2}|w^p|^2\) and
\(H_2=\frac{1}{2}(|u|^2-|v|^2)\). The simultaneous level sets
\[H_1=h_1,\quad H_2=h_2\] are tori whenever \(h_1\neq 0\) and
\((h_1,h_2)\neq\left(\frac{1}{2}n^2,0\right)\) for
\(n\in\{1,\ldots,d\}\). The fibres with \(h_1=0\) are isotropic
circles, and the fibres with
\((h_1,h_2)=\left(\frac{1}{2}n^2,0\right)\) are pinched tori. These
pinched tori are stratified by isotropic strata: they are made up of
a pinch point (0-stratum) and a Lagrangian cylinder (2-stratum).

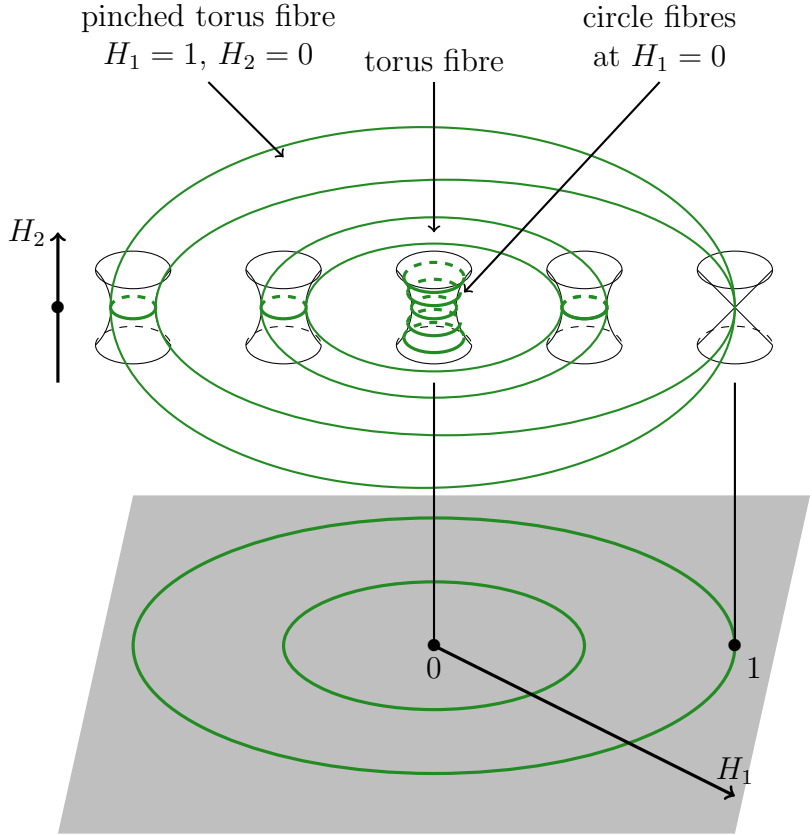
\begin{center}
\begin{tikzpicture}
\filldraw[fill=lightgray,opacity=0.5,draw=none] (-5,-7) -- (4,-7) -- (5,-2.5) -- (-4,-2.5) -- cycle;
\conic{(-4,0)};
\conic{(-2,0)};
\conic{(0,0)};
\conic{(2,0)};
\nodalconic{(4,0)};
\circdash{(0,0)}{0.3}{dg,very thick}{0};
\circdash{(0,0.2)}{0.35}{dg,very thick}{0};
\circdash{(0,0.4)}{0.4}{dg,very thick}{0};
\circdash{(0,-0.2)}{0.35}{dg,very thick}{0};
\circdash{(0,-0.4)}{0.4}{dg,very thick}{0};
\circdash{(2,0)}{0.3}{dg,very thick}{0};
\circdash{(-2,0)}{0.3}{dg,very thick}{0};
\draw[thick,dg] (0,0) circle [x radius=2.3,y radius=1.2];
\draw[thick,dg] (0,0) circle [x radius=1.7,y radius=0.85];
\circdash{(2,0)}{0.3}{dg,very thick}{0};
\circdash{(-4,0)}{0.3}{dg,very thick}{0};
\draw[thick,dg] (-0.15,0) circle [x radius=4.15,y radius=2.4];
\draw[thick,dg] (0.15,0) circle [x radius=3.85,y radius=1.7];
\draw[very thick,dg] (0,-4.5) circle [x radius=2,y radius=0.85];
\draw[very thick,dg] (0,-4.5) circle [x radius=4,y radius=1.7];
\draw[thick] (0,-1) -- (0,-4.5) node {\(\bullet\)} node [below] {\(0\)};
\draw[thick] (4,-1) -- (4,-4.5) node {\(\bullet\)} node [below right] {\(1\)};
\draw[<-,thick] (0.4,0.2) -- (3,3) node [align=center,above] {circle fibres\\ at \(H_1=0\)};
\draw[<-,thick] (-2,2) -- (-3,3) node [align=center,above] {pinched torus fibre\\ \(H_1=1\), \(H_2=0\)};
\draw[<-,thick] (0,1) -- (0,3) node [above] {torus fibre};
\draw[->,very thick] (0,-4.5) -- (4,-6.5) node [above] {\(H_1\)};
\draw[->,very thick] (-5,-1) -- (-5,0) node {\(\bullet\)} -- (-5,1) node [left] {\(H_2\)};

\end{tikzpicture}
\end{center}
\end{Example}
Here we draw (in the case \(d=3\)) the image of the map
\((H_1,H_2)\colon B_{d,p,q}\to\RR^2\), denoting the pinched torus
fibres with a cross:

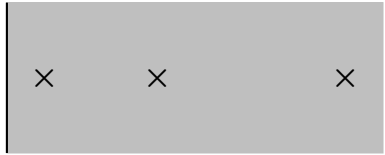
\begin{center}
\begin{tikzpicture}
\filldraw[lightgray] (0,-1) -- (0,1) -- (5,1) -- (5,-1) -- cycle;
\draw[thick] (0,-1) -- (0,1);
\node at (0.5,0) {\(\bm{\times}\)};
\node at (2,0) {\(\bm{\times}\)};
\node at (9/2,0) {\(\bm{\times}\)};

\end{tikzpicture}
\end{center}
(The picture should extend infinitely up, down and right.)

\begin{Remark}\label{rmk:almosttoric}
\begin{enumerate}
\item The fibres over the vertical boundary are circles; the preimage of
the vertical boundary is the conic at \(H_1=0\). The local model
for these kinds of singular fibre is precisely that of the
boundary of the moment polygon in a toric surface.

\item The singularities in the pinched torus fibres are called
``focus-focus singularities''. Lagrangian torus fibrations with
singularities modelled on focus-focus points and toric strata are
called {\em almost toric fibrations}.

\item The precise horizontal position of the crosses is not really very
important: it can be changed by varying the equation for the
smoothing.

\item The Lagrangian vanishing cycle we encountered last time is visible
in this picture: it projects to the green line in the figure
below.

\end{enumerate}
\end{Remark}
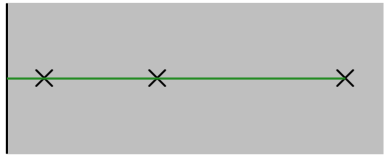
\begin{center}
\begin{tikzpicture}
\filldraw[lightgray] (0,-1) -- (0,1) -- (5,1) -- (5,-1) -- cycle;
\draw[thick] (0,-1) -- (0,1);
\node at (0.5,0) {\(\bm{\times}\)};
\node at (2,0) {\(\bm{\times}\)};
\node at (9/2,0) {\(\bm{\times}\)};
\draw[thick,dg] (0,0) -- (9/2,0);

\end{tikzpicture}
\end{center}
\subsubsection{Action coordinates}

This is not quite the end of the story. In toric geometry, we require
that the period lattice be \((2\pi\ZZ)^n\); without this condition, it
is not possible to reconstruct the toric variety uniquely from the
polytope. Indeed, as we saw in \cref{exm:quotsing}, if we don't impose
this condition, then all the singularities \(\frac{1}{n}(1,a)\) admit
``moment maps'' whose images are the nonnegative quadrant. In this
case, we needed to correct our naive moment image by a \(\QQ\)-affine
transformation.

In our current example (\(B_{d,p,q}\)), the problem is even more
severe: the Hamiltonian functions \(H_1,H_2\) define an
\(\RR^2\)-action, but the flow of \(H_1\) is not periodic: the
focus-focus fibres each comprise two orbits of the \(\RR^2\)-action, a
fixed point (the node) and a Lagrangian cylinder, on which the flow of
\(H_1\) acts by translation (hence without periods). Even for the
generic fibre, on which the flow of \(H_1\) is periodic, its period
varies from point to point.

We can fix this locally by using Hamiltonians of the form
\[(G(H_1,H_2),H_2).\] Near a smooth torus fibre, it is always possible
to find a \(G\) such that these modified Hamiltonians generate a torus
action with standard period lattice; the modified Hamiltonians are
called {\em action coordinates}. This determines \(G(H_1,H_2)\)
uniquely up to adding on an integral multiple of \(H_2\). It is
usually nontrivial to find \(G\) explicitly (even in simple
Hamiltonian systems like the pendulum, its expression turns out to
involve transcendental functions like elliptic integrals).

Globally, however, there are problems. For a start, \(G\) will always
be singular at the focus-focus fibres (there is no way to make a
non-periodic flow periodic). Cut out these focus-focus fibres and fix
our favourite smooth fibre. Solve for \(G\) locally near that fibre,
and prolong the solution to find \(G\) everywhere away from the
focus-focus fibres. If you prolong around a loop in the base which
encloses some focus-focus fibres, then there is no guarantee that
\(G\) remains single-valued; the difference in values will be an
integral multiple of \(H_2\). In fact, a local computation shows that
if you go once anticlockwise around a focus-focus fibre then the
Hamiltonians \((G(H_1,H_2),H_2)\) become \((G(H_1,H_2)-H_2,H_2)\). In
other words, the {\em monodromy} for a focus-focus fibre is
\(M=\begin{pmatrix} 1 & -1 \\ 0 & 1\end{pmatrix}\). See the example of
monodromy for the spherical pendulum in \cite{Duistermaat} (he obtains
the transpose-inverse of our monodromy, because he is interested in
how monodromy acts on the homology of the torus fibres).

To combat this ambiguity, we will make branch cuts in the base of our
torus fibration so as to remove the focus-focus fibres and ensure that
the complement of the branch cuts is simply-connected. Then the action
coordinates are well-defined on this domain. Without further
justification, I will draw for you the result:

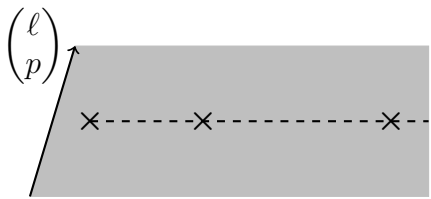
\begin{center}
\begin{tikzpicture}
\filldraw[lightgray] (-0.3,-1) -- (0.3,1) -- (5,1) -- (5,-1) -- cycle;
\draw[->,thick] (-0.3,-1) -- (0.3,1);
\node at (0.3,1) [left] {\(\begin{pmatrix}\ell \\p\end{pmatrix}\)};
\node at (0.5,0) {\(\bm{\times}\)};
\node at (2,0) {\(\bm{\times}\)};
\node at (9/2,0) {\(\bm{\times}\)};
\draw[dashed,thick] (0.5,0) -- (5,0);

\end{tikzpicture}
\end{center}
The slanted edge points in the direction \((\ell,p)\) where
\(\ell\in\{0,\ldots,p-1\}\) is the multiplicative inverse of \(q\) mod
\(p\). The branch cuts are drawn as dotted lines. Changing coordinates
using the \(\ZZ\)-affine transformation \(N=\begin{pmatrix} p & -\ell
\\ q & -k\end{pmatrix}\) (where \(kp+q\ell=1\)) we get an alternative
picture:

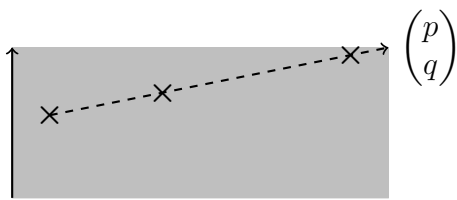
\begin{center}
\begin{tikzpicture}
\filldraw[lightgray] (0,-1) -- (0,1) -- (5,1) -- (5,-1) -- cycle;
\draw[->,thick] (0,-1) -- (0,1);
\node at (0.5,0.1) {\(\bm{\times}\)};
\node at (2,0.4) {\(\bm{\times}\)};
\node at (9/2,0.9) {\(\bm{\times}\)};
\draw[->,dashed,thick] (0.5,0.1) -- (5,1) node [right] {\(\begin{pmatrix} p \\ q\end{pmatrix}\)};

\end{tikzpicture}
\end{center}
Here are some remarks:
\begin{itemize}
\item We have made our branch cuts in directions collinear with the
eigenvector of the monodromy. If we had chosen another direction,
our diagram would have been a sector with two dotted lines, related
by the monodromy, and whenever you leave the sector by crossing one
dotted line, you reappear at the other one with your tangent vector
twisted by the monodromy matrix.
\item In the latest picture, the monodromy matrix \(M\) has changed to its
conjugate \(NMN^{-1}=\begin{pmatrix}1+pq & -p^2 \\ q^2 &
1-pq\end{pmatrix}\) because we changed coordinates using \(N\). It
is easy to check that the branch cuts (now in direction \((p,q)\))
are still pointing in the eigendirection of this matrix.

\end{itemize}
\subsubsection{Mutation}

One final modification will let us compare the toric picture of
\(\frac{1}{dp^2}(1,dpq-1)\) and the almost toric picture of its
smoothing directly. We simply use a different set of branch cuts,
pointing in the negative eigendirection. These branch cuts now
intersect the toric boundary, which means that the vertical straight
line in our earlier pictures appears bent in the new
pictures. Nonetheless, it is still a straight line because when we
cross the branch cut we must apply the monodromy to our tangent
vectors. Because we are crossing \(d\) branch cuts, each with
monodromy \(\begin{pmatrix} 1+pq & -p^2 \\ q^2 & 1-pq\end{pmatrix}\),
the tangent vector \((0,-1)\) of our vertical line (travelling
downwards) is sent to \((dp^2,dpq-1)\).

Here is the result; I will refer to this almost toric diagram as
\(\tilde{\pi}(dp^2,dpq-1)\).

\begin{center}
\begin{tikzpicture}
\filldraw[lightgray] (0,2) -- (0,0) node[black] {\(\circ\)} -- (10,0.8) -- (10,2) -- cycle;
\draw[->,thick] (0,2) -- (0,0) -- (10,0.8) node [below] {\(\begin{pmatrix} dp^2 \\ dpq-1\end{pmatrix}\)};
\node at (1,0.2) {\(\bm{\times}\)};
\node at (4,0.8) {\(\bm{\times}\)};
\node at (9,1.8) {\(\bm{\times}\)};
\draw[dashed,thick] (0,0) -- (9,1.8) node [midway,above] {\(\begin{pmatrix} p \\ q\end{pmatrix}\)};
\node at (1.5,1) {\(\tilde{\pi}(dp^2,dpq-1)\)};

\end{tikzpicture}
\end{center}
If we were to erase the branch cuts and decorations indicating the
focus-focus fibres, we would obtain the standard moment polygon for
the toric \(\frac{1}{dp^2}(1,dpq-1)\) singularity.

\begin{Remark}
This trick of changing branch cut by 180 degrees is called {\em
mutation}. To perform a mutation, we slice our picture in two using
the branch cut, we apply the monodromy to one of the two halves and
then reattach them. This will be an important operation in what
follows. Note that mutation does not change the symplectic manifold
or even the torus fibration, it only changes the choice of branch
cut used in producing the action coordinates.

\end{Remark}
\subsection{Smoothings of other cyclic quotient singularities}
\label{sct:ksb}

We can now draw almost toric pictures for smoothings of other cyclic
quotient singularities.

\subsubsection{Koll\'{a}r--Shepherd-Barron classifications}

The smoothings of cyclic quotient singularities were studied by
Looijenga and Wahl \cite{LooijengaWahl} and by Koll\'{a}r and
Shepherd-Barron \cite{KSB} and can be completely
classified. Here is a sketch of the KSB approach to classifying
smoothings:
\begin{itemize}
\item Take a smoothing \(\mathcal{X}\to\CC\) where \(X_0=\CC^2/G\)
is the singularity of type \(\frac{1}{n}(1,a)\). Write
\(x_0\in X_0\) for the singular point.
\item The total space \(\mathcal{X}\) is usually singular at \(x_0\). By
the semistable reduction theorem, you can perform base-change and
birationally modify \(\mathcal{X}\) to get a new smoothing
\(\mathcal{X}'\to\CC\) with smooth total space, and to make \(X'_0\)
into a reduced simple normal crossing divisor.
\item Using the semistable minimal model program, you can further modify
\(\mathcal{X}\) to get a smoothing \(\mathcal{X}''\to\CC\) so that
any curve \(C\) contracted by the birational map
\(\mathcal{X}''\to\mathcal{X}\) satisfies \(K_{\mathcal{X}''}\cdot
C>0\). The cost of this step is that you may introduce terminal
singularities in the total space.
\item The cyclic quotient singularities whose smoothings have at
worst isolated terminal singularities are classified by
Looijenga and Wahl \cite{LooijengaWahl} and Koll\'{a}r and
Shepherd-Barron \cite{KSB}; they are precisely the
T-singularities we have already discussed.
\end{itemize}
Therefore, after all of this messing around, we have a central fibre
\(X''_0\) which is a partial resolution of the original \(X_0\),
obtained by blowing up some sequence of points. The singularities of
\(X''_0\) are T-singularities and the canonical class evaluates
nonnegatively on curves. We finally take the canonical model; this may
introduce some ADE singularities, but ensures that the canonical class
evaluates positively on all curves.

\begin{Definition}[P-resolution]\label{dfn:pres}
A partial resolution \(f\colon Y_0\to X_0\) of a cyclic quotient
singularity is called a {\em P-resolution} if:
\begin{itemize}
\item it has only T-singularities,
\item \(K_{Y_0}\) evaluates positively on all exceptional curves of
\(f\).

\end{itemize}
\end{Definition}
\subsubsection{Toric P-resolutions and their almost toric smoothings}

In terms of toric pictures, partial resolution is accomplished by
making truncations to the moment polygon (just as blowing up a smooth
toric fixed point amounts to chopping off a corner of the moment
polytope). It turns out that there are only a finite number of
truncations giving a P-resolution (for an explanation of this, and how
to find the P-resolutions, see {\cite[Section 3]{KSB}}; I have given a
brief, pragmatic account in \cref{app:pres}).

\begin{Example}[{\cite[Example 3.15]{KSB}}]\label{exm:KSBpres}
There are three P-resolutions of the singularity
\(\frac{1}{19}(1,7)\). In each case, we draw the toric model and
indicate the singular points as red dots. I have also written the
type of the singularity (not to be confused with the slope-vectors
labelling the edges).

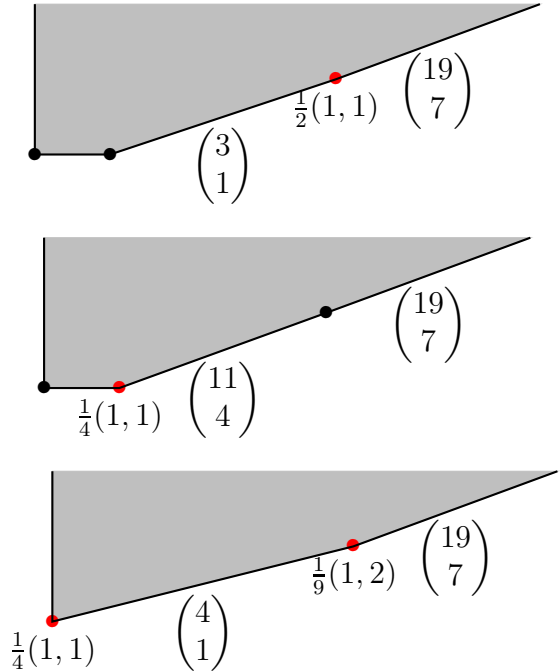
\begin{center}
\begin{tikzpicture}
\filldraw[lightgray] (0,2) -- (0,0) node[black] {\(\bullet\)} -- (1,0) node[black] {\(\bullet\)} -- (4,1) node[red] {\(\bullet\)} -- (4+19/7,2) -- cycle;
\draw[thick] (0,2) -- (0,0) -- (1,0) -- (4,1) -- (4+19/7,2);
\node at (2.5,0.5) [below] {\(\begin{pmatrix}3 \\ 1\end{pmatrix}\)};
\node at (4+19/14,1.5) [below] {\(\begin{pmatrix}19 \\ 7\end{pmatrix}\)};
\node at (4,1) [below] {\(\frac{1}{2}(1,1)\)};
\end{tikzpicture}
\end{center}
\begin{center}
\begin{tikzpicture}
\filldraw[lightgray] (0,2) -- (0,0) node[black] {\(\bullet\)} -- (1,0) node[red] {\(\bullet\)} -- (1+11/4,1) node[black] {\(\bullet\)} -- (1+11/4+19/7,2) -- cycle;
\draw[thick] (0,2) -- (0,0) -- (1,0) -- (1+11/4,1) -- (1+11/4+19/7,2);
\node at (1+11/8,0.5) [below] {\(\begin{pmatrix}11 \\ 4\end{pmatrix}\)};
\node at (1+11/4+19/14,1.5) [below] {\(\begin{pmatrix}19 \\ 7\end{pmatrix}\)};
\node at (1,0) [below] {\(\frac{1}{4}(1,1)\)};
\end{tikzpicture}
\end{center}
\begin{center}
\begin{tikzpicture}
\filldraw[lightgray] (0,2) -- (0,0) node[red] {\(\bullet\)} -- (4,1) node[red] {\(\bullet\)} -- (4+19/7,2) -- cycle;
\draw[thick] (0,2) -- (0,0) -- (4,1) -- (4+19/7,2);
\node at (2,0.5) [below] {\(\begin{pmatrix}4 \\ 1\end{pmatrix}\)};
\node at (4+19/14,1.5) [below] {\(\begin{pmatrix}19 \\ 7\end{pmatrix}\)};
\node at (0,0) [below] {\(\frac{1}{4}(1,1)\)};
\node at (4,1) [below] {\(\frac{1}{9}(1,2)\)};

\end{tikzpicture}
\end{center}
Each singular point is a T-singularity and can be smoothed by
implanting the local almost toric model at the relevant vertex. We
must be careful to use a \(\ZZ\)-affine transformation to transfer
the almost toric diagram \(\tilde{\pi}(dp^2,dpq-1)\) to the relevant
vertex.

In the first picture, for example, we use \(A=\begin{pmatrix}11 & -3
\\ 4 & -1\end{pmatrix}\) to put the vector \((0,1)\) into the
\((-3,-1)\)-direction and \((2,1)\) into the
\((19,7)\)-direction. The branch cuts now must point in the
\(A\begin{pmatrix} 1 \\ 1\end{pmatrix}=\begin{pmatrix} 8
\\ 3\end{pmatrix}\)-direction. Here are the resulting diagrams (not
all are drawn to scale because some branch cuts are almost parallel
to the boundary):
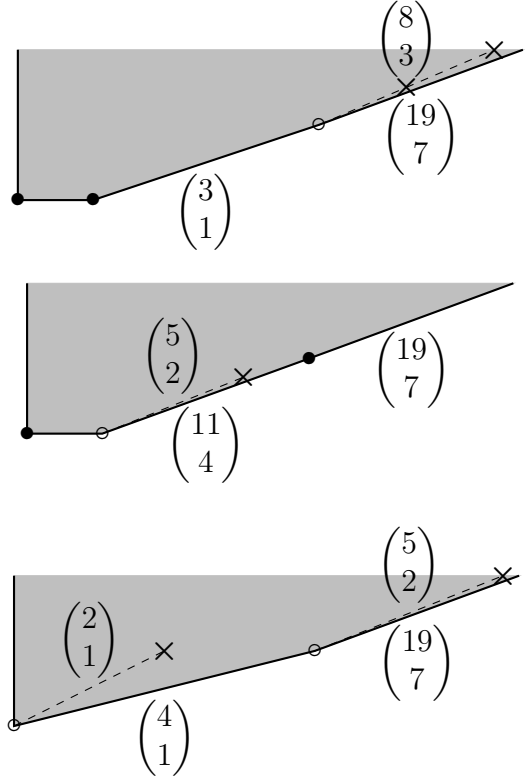
\begin{center}
\begin{tikzpicture}
\filldraw[lightgray] (0,2) -- (0,0) node[black] {\(\bullet\)} -- (1,0) node[black] {\(\bullet\)} -- (4,1) node[black] {\(\circ\)} -- (4+19/7,2) -- cycle;
\draw[thick] (0,2) -- (0,0) -- (1,0) -- (4,1) -- (4+19/7,2);
\node at (2.5,0.5) [below] {\(\begin{pmatrix}3 \\ 1\end{pmatrix}\)};
\node at (4+19/14,1.5) [below] {\(\begin{pmatrix}19 \\ 7\end{pmatrix}\)};
\draw[dashed] (4,1) -- (4+7/6,1+1/2) node {\(\bm{\times}\)} -- (4+7/3,1+1) node {\(\bm{\times}\)};
\node at (4+7/6,1+1/2) [above] {\(\begin{pmatrix} 8 \\ 3\end{pmatrix}\)};
\end{tikzpicture}
\end{center}
\begin{center}
\begin{tikzpicture}
\filldraw[lightgray] (0,2) -- (0,0) node[black] {\(\bullet\)} -- (1,0) node[black] {\(\circ\)} -- (1+11/4,1) node[black] {\(\bullet\)} -- (1+11/4+19/7,2) -- cycle;
\draw[thick] (0,2) -- (0,0) -- (1,0) -- (1+11/4,1) -- (1+11/4+19/7,2);
\node at (1+11/8,0.5) [below] {\(\begin{pmatrix}11 \\ 4\end{pmatrix}\)};
\node at (1+11/4+19/14,1.5) [below] {\(\begin{pmatrix}19 \\ 7\end{pmatrix}\)};
\draw[dashed] (1,0) -- (1+15/8,3/4) node {\(\bm{\times}\)};
\node at (1+15/16,3/8) [above] {\(\begin{pmatrix} 5 \\ 2\end{pmatrix}\)};
\end{tikzpicture}
\end{center}
\begin{center}
\begin{tikzpicture}
\filldraw[lightgray] (0,2) -- (0,0) node[black] {\(\circ\)} -- (4,1) node[black] {\(\circ\)} -- (4+19/7,2) -- cycle;
\draw[thick] (0,2) -- (0,0) -- (4,1) -- (4+19/7,2);
\node at (2,0.5) [below] {\(\begin{pmatrix}4 \\ 1\end{pmatrix}\)};
\node at (4+19/14,1.5) [below] {\(\begin{pmatrix}19 \\ 7\end{pmatrix}\)};
\draw[dashed] (0,0) -- (2,1) node {\(\bm{\times}\)};
\node at (1,0.5) [above] {\(\begin{pmatrix} 2 \\ 1\end{pmatrix}\)};
\draw[dashed] (4,1) -- (4+5/2,2) node {\(\bm{\times}\)};
\node at (4+5/4,1.5) [above] {\(\begin{pmatrix} 5 \\ 2\end{pmatrix}\)};

\end{tikzpicture}
\end{center}
(I've drawn the smoothed vertices as open circles, to distinguish
them from the actual vertices: remember that the `broken edge'
passing through a smoothed vertex is really a straight line). In
particular, we see that:
\begin{itemize}
\item the first example contains a symplectically embedded \(B_{2,1,1}\)
(neighbourhood of a Lagrangian sphere),
\item the second example contains a symplectically embedded \(B_{2,1}\)
(neighbourhood of a Lagrangian \((2,1)\)-pinwheel, i.e.
\(\rp{2}\).
\item the third example contains a symplectically embedded \(B_{2,1}\)
and a symplectically embedded \(B_{3,1}\) (neighbourhood of a
Lagrangian \(\rp{2}\) and a \((3,1)\)-pinwheel).
\end{itemize}
These manifest themselves in our almost toric pictures as preimages
under the torus fibration of neighbourhoods of the dashed lines.

\end{Example}
\begin{Remark}
Consider the lens space \(L(n,a)\) equipped with the contact
structure it inherits as the link of \(\frac{1}{n}(1,a)\). The
minimal symplectic fillings of this contact manifold were classified
up to diffeomorphism by Lisca \cite{Lisca}. It was proved in
\cite{NPP} that his classification coincides with that of
Koll\'{a}r--Shepherd-Barron; that is any minimal symplectic filling
is diffeomorphic to a smoothing of the corresponding cyclic quotient
singularity and smoothings coming from different P-resolutions are
not diffeomorphic (it seems likely that this classification holds up
to the stronger equivalence relation given by symplectic
deformation/diffeomorphism, but I believe this statement is still
open\footnote{Added later: At the KIAS conference, I was made
aware of the work of Bhupal and Ono \cite{BhupalOno} who prove
this stronger result.}).

\end{Remark}
\newpage

\section{Lecture 3}

\subsection{Applications}

Next, I want to review some of the theorems you can prove using
symplectic techniques to put constraints on singularity formation.

\begin{Theorem}[Evans-Smith \cite{ES1}]\label{thm:cp2es}
Suppose that we have \(N\) disjointly and symplectically embedded
submanifolds \(U_1,\ldots,U_N\subset\cp{2}\) where each \(U_i\) is a
copy of \(B_{p_i,q_i}\). Then \(N\leq 3\) and
\(\{p_1,p_2,\ldots,p_N\}\) is a subset of a Markov triple, that is a
solution \(\{a,b,c\}\) to the Diophantine equation
\(a^2+b^2+c^2=3abc\) (moreover, the \(q\)s are determined by the
\(p\)s in the triple).

\end{Theorem}
Note that \(d=1\) for any \(B_{d,p,q}\subset\cp{2}\) because
\(\cp{2}\) contains no Lagrangian spheres (just for homology reasons).

\begin{Remark}
This theorem is a symplectic geometer's translation of the
following, earlier theorem of Hacking and Prokhorov, which inspired
it:

\end{Remark}
\begin{Theorem}[Hacking-Prokhorov \cite{HP1,HP2}]\label{thm:cp2hp}
Suppose \(\mathcal{X}\to S\) is a \(\QQ\)-Gorenstein degeneration
whose general fibre is \(\cp{2}\) and whose singular fibre \(X_0\)
has ample anticanonical bundle and at worst isolated quotient
singularities. Then \(X_0\) is a (\(\QQ\)-Gorenstein deformation of)
a weighted projective space \(P(p_1^2,p_2^2,p_3^2)\) where
\(\{p_1,p_2,p_3\}\) is a Markov triple. (These weighted projective
spaces have precisely three singularities of type
\(\frac{1}{p_i^2}(1,p_iq_i-1)\)).

\end{Theorem}
\begin{Theorem}[Evans-Smith \cite{ES2}]\label{thm:gtes}
Let \(X\) be a smooth surface of general type with \(p_g>0\)
(\(b^+>1\)) equipped with its canonical symplectic form and
suppose it contains a symplectically embedded copy of
\(B_{p,q}\) for some \(p,q\). Let \(\ell\) be the length of
the continued fraction expansion
\[\frac{p^2}{pq-1}=b_1-\frac{1}{b_2-\frac{1}{\cdots-\frac{1}{b_\ell}}}.\]
Then\footnote{Added later: When you have an exceptional curve
of the first kind, the irreducible components come with
multiplicities so that their sum has square \(-1\). The proof
in Evans--Smith didn't make use of these multiplicities; if
you take them into account then you can probably push the
bound up to the known optimal bound of \(4K^2+1\).} \(\ell\leq
4K^2+7\).

\end{Theorem}
\begin{Remark}
The canonical symplectic form is obtained as follows. The canonical
bundle of \(X\) is nef and big, and vanishes only along a collection
of embedded rational\(-2\)-curves. The (pluri)canonical map
contracts these \(-2\)-curves, so the canonical model is a surface
with ADE singularities. We can now perform symplectic surgery,
cutting out these singularities and replacing them with copies of
the ADE Milnor fibres. The result is a smooth symplectic manifold
whose symplectic form is {\em negatively monotone}, i.e. it
satisfies \([\omega]=K_X\).

\end{Remark}
\begin{Remark}
This theorem implies the corresponding bound for lengths of
singularities in stable degenerations of general type surfaces (in a
stable degeneration, the central fibre has ample canonical bundle,
which means that the canonical symplectic form makes sense over the
whole family, allowing us to define symplectic parallel
transport). This bound (in fact, a slightly better bound of
\(\ell\leq 4K^2+1\)) was proved independently and simultaneously
using purely algebro-geometric techniques in a paper by Rana and
Urz\'{u}a \cite{RU}.

\end{Remark}
In this lectures, I will not discuss these theorems any further,
because the proofs require techniques of pseudoholomorphic curve
theory and Seiberg-Witten theory which I would not have time to
cover. Instead, I will try to explain the proof of the following
theorem, and what it has to do with algebraic geometry (specifically
the minimal model program).

\begin{Theorem}[Evans-Urz\'{u}a \cite{EU}]\label{thm:eu}
Let \(X\) be the quintic surface. There exists a symplectic form
\(\omega\) on \(X\) which is a deformation of its canonical
symplectic form and such that \((X,\omega)\) contains symplectically
embedded copies of \(B_{p_i,q_i}\) for \(p_i,q_i\) on the following
list: \[(5,3),\ (14,9),\ (37,24),\ \ldots\] This is an infinite
list, where \(p_i\) and \(q_i\) both satisfy the recursion relation
\[\begin{pmatrix} 0 & 1 \\ -1 & \delta\end{pmatrix}\begin{pmatrix}
p_i \\ p_{i+1}\end{pmatrix}=\begin{pmatrix} p_{i+1}
\\ p_{i+2}\end{pmatrix},\qquad\delta=p_iq_{i+1}-p_{i+1}q_i=3.\]

\end{Theorem}
The reason we are allowed to violate the inequality \(\ell\leq
4K^2+7\) is because the symplectic form is not the canonical one. It
would be interesting to investigate in greater detail how the
symplectic embeddability of \(B_{p,q}\)s depends on the cohomology
class of \(\omega\).

\begin{Remark}
Though this theorem is about quintic surfaces, we have other
examples, and indeed we expect that such sequences of symplectic
embeddings of unbounded length can be constructed in symplectic
deformations of most general type surfaces.

\end{Remark}
\begin{Remark}
The proof of \cref{thm:eu} draws heavily on ideas from the paper
\cite{HTU} of Hacking, Tevelev and Urz\'{u}a. Indeed, our paper
\cite{EU} can be seen as a symplectic account of what is happening
in \cite{HTU}.

\end{Remark}
We will start our proof of \cref{thm:eu} with an example.

\subsection{An example: the \(\frac{1}{11}(1,3)\) singularity}

Consider the
cyclic quotient singularity \(\frac{1}{11}(1,3)\).

\begin{center}
\begin{tikzpicture}
\filldraw[draw=black,thick,->,fill=lightgray] (0,3) -- (0,0) -- (11,3);
\draw[step=1.0,black,thin,dotted] (0,0) grid (11,3);

\end{tikzpicture}
\end{center}
We will truncate this polygon with a line pointing in the
\((1,-2)\)-direction. Let us call the resulting polygon \(\Pi^-\):

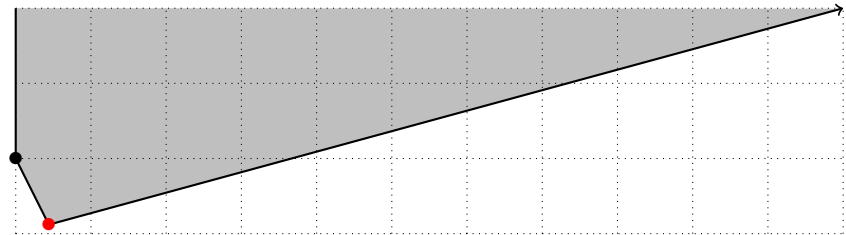
\begin{center}
\begin{tikzpicture}
\filldraw[draw=black,thick,->,fill=lightgray] (0,3) -- (0,1) node[black] {\(\bullet\)} -- (11/25,3/25) node[red] {\(\bullet\)} -- (11,3);
\draw[step=1.0,black,thin,dotted] (0,0) grid (11,3);

\end{tikzpicture}
\end{center}
The associated toric variety \(V_{\Pi^-}\) is {\bf not} a P-resolution
of \(\frac{1}{11}(1,3)\): if you evaluate the canonical class on the
curve which lives over the line we introduced by truncation then you
get a negative number (see \cref{app:pres}). Nonetheless, we can check
that \(V_{\Pi^-}\) has only T-singularities:
\begin{itemize}
\item the left-hand vertex is actually a smooth point: the outgoing edges
point in the directions \((0,1)\) and \((1,-2)\), which form an
integral basis for the lattice \(\ZZ^2\).
\item the right-hand (red) vertex is a \(\frac{1}{25}(1,14)\)-singularity,
whose smoothing will be a \(B_{5,3}\). To see this, note that the
matrix \(\begin{pmatrix}1 & -1 \\ -1 & 2\end{pmatrix}\) sends
\((0,1)\) and \((25,14)\) (the outgoing edges in the standard
polygon \(\tilde{\pi}(25,14)\)) to \((-1,2)\) and \((11,3)\)
respectively, which are the outgoing edges at this vertex. The
branch cut points in the direction \[\begin{pmatrix}1& -1 \\ -1 &
2\end{pmatrix}\begin{pmatrix}5\\3\end{pmatrix}=\begin{pmatrix}2
\\ 1\end{pmatrix}.\]
\end{itemize}
Let us write \(U_{\Pi^-}\) for the smoothing of \(V_{\Pi^-}\). From
what we just said, \(U_{\Pi^-}\) has the following almost toric
picture, in which we can see a symplectically embedded \(B_{5,3}\)
living over the blue-shaded region:

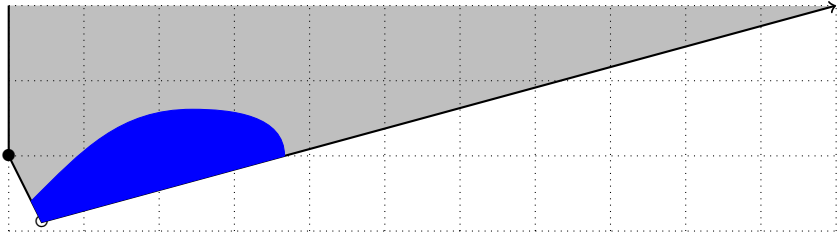
\begin{center}
\begin{tikzpicture}
\filldraw[draw=black,thick,->,fill=lightgray] (0,3) -- (0,1) node[black] {\(\bullet\)} -- (11/25,3/25) node[black] {\(\circ\)} -- (11,3);
\draw[step=1.0,black,thin,dotted] (0,0) grid (11,3);
\draw[dashed] (11/25,3/25) -- (11/25+2,3/25+1) node {\(\bm{\times}\)};
\filldraw[opacity=0.3,blue] (11/25,3/25) -- (11/3,1) to[out=90,in=0] (11/25+2,3/25+1.5) to[out=180,in=45] (0.3,1-0.6) -- cycle;

\end{tikzpicture}
\end{center}
By Lisca's classification, this must be diffeomorphic to one of the
smoothings which comes from a P-resolution, but how can we identify
which one?

First, let us perform a mutation and rotate the branch cut
anticlockwise by 180 degrees; we know that the broken edge which used
to meet the branch cut will become a straight line after this
mutation, so we don't need to think about it too much.

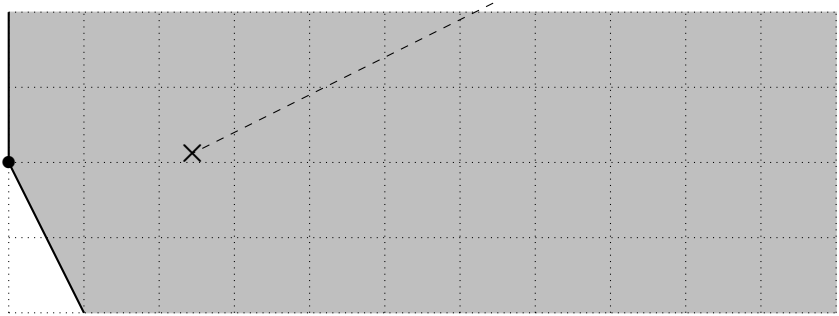
\begin{center}
\begin{tikzpicture}
\filldraw[lightgray] (0,3) -- (0,1) node[black] {\(\bullet\)} -- (11/25,3/25) -- (1,-1) -- (11,-1) -- (11,3) -- cycle;
\draw[draw=black,thick] (0,3) -- (0,1) -- (11/25,3/25) -- (1,-1);
\draw[step=1.0,black,thin,dotted] (0,-1) grid (11,3);
\draw[dashed] (11/25+6,3/25+3) -- (11/25+2,3/25+1) node {\(\bm{\times}\)};

\end{tikzpicture}
\end{center}
Now I want to change the picture by moving the focus-focus singularity
and the branch cut in parallel to the north-west.

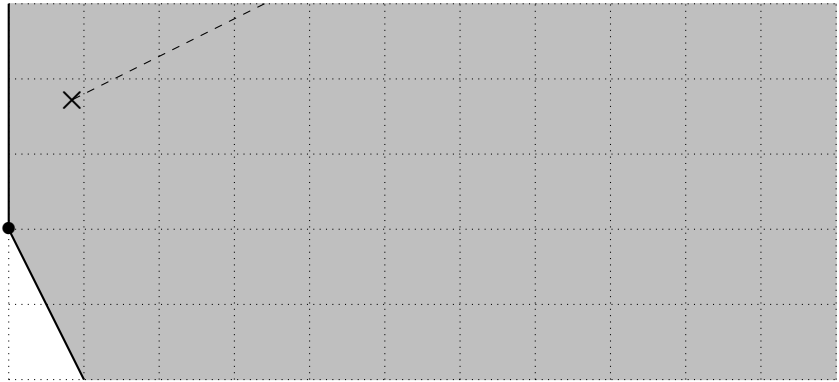
\begin{center}
\begin{tikzpicture}
\filldraw[lightgray] (0,4) -- (0,1) node[black] {\(\bullet\)} -- (1,-1) -- (11,-1) -- (11,4) -- cycle;
\draw[draw=black,thick] (0,4) -- (0,1) -- (11/25,3/25) -- (1,-1);
\draw[step=1.0,black,thin,dotted] (0,-1) grid (11,4);
\draw[dashed] (7/5+2,4) -- (11/25+2-1.6,3/25+1+1.6) node {\(\bm{\times}\)};

\end{tikzpicture}
\end{center}
Finally, we mutate back, rotating the branch cut 180 degrees
clockwise. We need to think about this: the monodromy for the branch
cut in \(\tilde{\pi}(25,14)\) is
\(M=\begin{pmatrix}16&-25\\ 9&-14\end{pmatrix}\); we used
\(N=\begin{pmatrix}1 & -1 \\ -1 & 2\end{pmatrix}\) to put
\(\tilde{\pi}(25,14)\) into position at the vertex in our picture, so
the monodromy matrix becomes
\(NMN^{-1}=\begin{pmatrix}3&-4\\1&-1\end{pmatrix}\). Applying this to
the red shaded region in the picture, we need to see where the vectors
\((0,-1)\) and \((1,-2)\) end up (as these are the slopes of the edges
in this region).

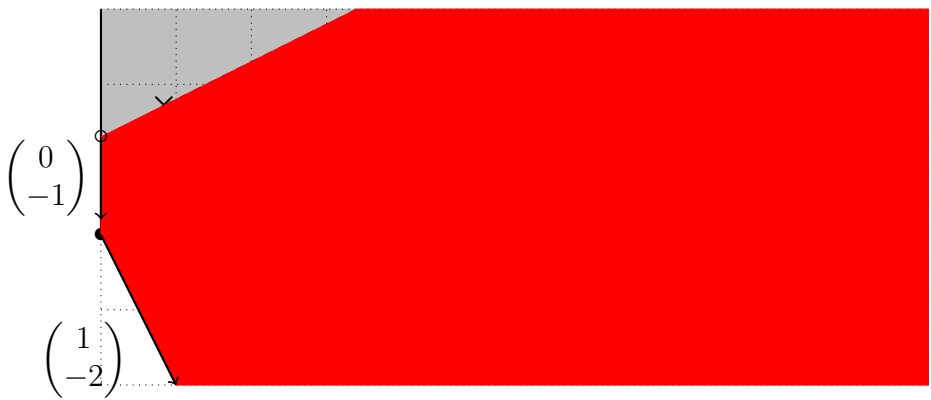
\begin{center}
\begin{tikzpicture}
\filldraw[lightgray] (0,4) -- (0,1) node[black] {\(\bullet\)} -- (1,-1) -- (11,-1) -- (11,4) -- cycle;
\draw[draw=black,thick] (0,4) -- (0,1) -- (11/25,3/25) -- (1,-1);
\draw[step=1.0,black,thin,dotted] (0,-1) grid (11,4);
\draw[dashed] (7/5+2,4) -- (11/25+2-1.6,3/25+1+1.6) node {\(\bm{\times}\)};
\filldraw[red,opacity=0.5] (0,3/25+1+1.6-21/50) -- (0,1) -- (1,-1) -- (11,-1) -- (11,4) -- (7/5+2,4) -- cycle;
\draw[thick,->] (0,3/25+1+1.6-21/50) node {\(\circ\)} -- (0,1.2) node[midway,left] {\(\begin{pmatrix}0\\-1\end{pmatrix}\)};
\draw[thick,->] (0,1) -- (1,-1) node[midway,below left] {\(\begin{pmatrix}1\\-2\end{pmatrix}\)};

\end{tikzpicture}
\end{center}
They go to \((4,1)\) and \((11,3)\), so we end up
with the diagram below (not drawn to scale: \((4,1)\) and \((11,3)\)
are very close to parallel).

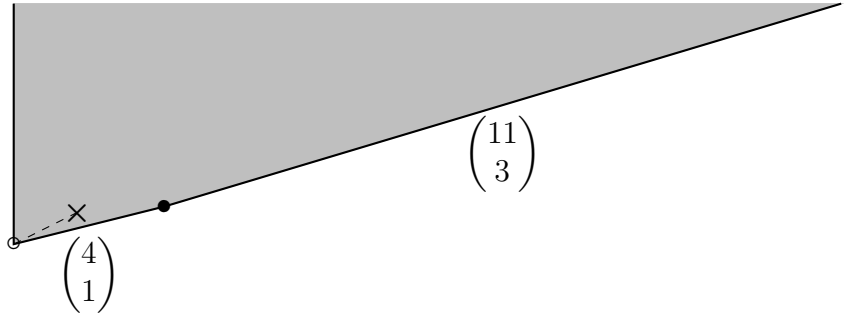
\begin{center}
\begin{tikzpicture}
\filldraw[lightgray] (0,5.5) -- (0,2.3) node[black] {\(\circ\)} -- (2,2.3+0.5) node [black] {\(\bullet\)} -- (11,5.5) -- cycle;
\draw[draw=black,thick] (0,5.5) -- (0,2.3) -- (2,2.8) -- (11,5.5);
\draw[dashed] (0,3/25+1+1.6-21/50) -- (11/25+2-1.6,3/25+1+1.6) node {\(\bm{\times}\)};
\node at (1,2.3+0.25) [below] {\(\begin{pmatrix}4\\1\end{pmatrix}\)};
\node at (6.5,4.15) [below] {\(\begin{pmatrix}11\\3\end{pmatrix}\)};

\end{tikzpicture}
\end{center}
Let us call the resulting polygon \(\Pi^+\), the corresponding toric
variety \(V_{\Pi^+}\) and the almost toric manifold determined by the
picture above \(U_{\Pi^+}\) (so \(U_{\Pi^+}\) is a smoothing of
\(V_{\Pi^+}\)). The variety \(V_{\Pi^+}\) is a P-resolution of
\(\frac{1}{11}(1,3)\), which tells us whereabouts in the
Koll\'{a}r--Shepherd-Barron/Lisca classification the smooth 4-manifold
\(U_{\Pi^-}\) lives.

\begin{Remark}
How do you know if a given toric diagram defines a P-resolution?
Once you've checked that the vertices are locally isomorphic to
wedges of the form \(\pi(dp^2,dpq-1)\), you still need to check that
the canonical class evaluates positively on the rational curves
which make up the compact part of the toric divisor. I endeavour to
explain how to check this in an appendix to these notes.

\end{Remark}
\subsection{Moving the singularity}

In the previous example, we performed an operation which we have not
yet justified. Namely, we deformed the almost toric picture by moving
the focus-focus singularity and its branch cut in parallel to the
north-west.

We have already mentioned (\cref{rmk:almosttoric}(3), see also
{\cite[Proposition 6.2]{Symington}}) that you can move the focus-focus
singularity in the direction which is an eigenvector for its monodromy
(i.e. along its branch cut, if you have chosen your branch cut to
point along this eigenvector); this changes the Lagrangian torus
fibration, but it does not change the total space up to
symplectomorphism. In the language of Symington \cite{Symington}, this
is a {\em nodal slide}.

However, if we wish to move the focus-focus singularity in a different
direction, we will need to deform the symplectic form in a more
serious way. Any almost toric picture with contractible
base\footnote{If the base has nontrivial topology then there is an
additional {\em Lagrangian Chern class} which can distinguish between
fibrations over this base, \cite{Zung}.} specifies for us a symplectic
manifold {\cite[Corollary 5.4]{Symington}} unique up to
symplectomorphism. This works in much the same way that it does in the
toric setting, except that we need to be slightly careful about how to
handle the focus-focus singularities. The resulting manifold is only
determined up to symplectomorphisms which do not necessarily preserve
the Lagrangian torus fibrations (while a toric moment image determines
the manifold up to equivariant symplectomorphisms, which preserve the
torus fibrations).

Our one-parameter family of almost toric pictures gives us a
one-parameter family \(\omega_t\) of symplectic structures, but
\(\omega_0\) and \(\omega_1\) will not be symplectomorphic:

\begin{Lemma}
In a family \(\omega_t\) of symplectic forms on \(U_{\Pi^-}\)
determined by the above sequence of almost toric pictures,
\([\omega_0]\) is a negative multiple of the canonical class \(K\)
and \([\omega_1]\) is a positive multiple of \(K\).
\end{Lemma}
\begin{proof}
Note that \(H_2(U_{\Pi^-};\ZZ)=\ZZ\); we can see a generator as
follows. Consider the almost toric picture at the beginning of the
sequence:
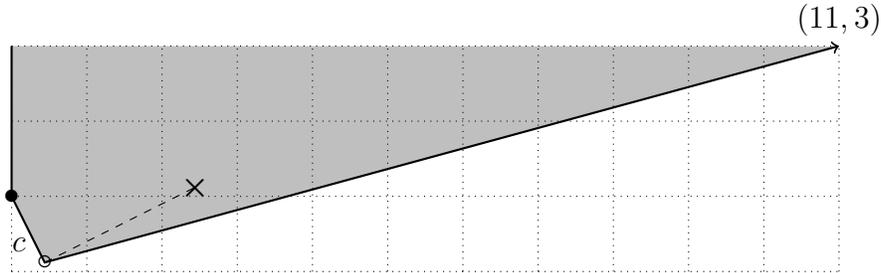
\begin{center}
\begin{tikzpicture}
\filldraw[draw=black,thick,->,fill=lightgray] (0,3) -- (0,1) node[black] {\(\bullet\)} -- (11/25,3/25) node[black] {\(\circ\)} -- (11,3);
\draw[step=1.0,black,thin,dotted] (0,0) grid (11,3);
\node at (11,3) [above] {\((11,3)\)};
\draw[dashed] (11/25,3/25) -- (11/25+2,3/25+1) node {\(\bm{\times}\)};
\node at (0.1,0.6) [below] {\(c\)};
\end{tikzpicture}
\end{center}
Consider the following chains:
\begin{itemize}
\item \(C\): the preimage of the edge \(c\) in the
almost toric picture
\item \(L\): the Lagrangian pinwheel living
over the branch cut.
\end{itemize}
These are both discs with a common boundary (the circular fibre over
the vertex marked with a circle), however, the pinwheel wraps
\(p=5\) times around this circle, so to get a closed cycle we must
take \(G_{\Pi^-}=5C-L\). This cycle \(G_{\Pi^-}\) generates
\(H_2(U_{\Pi^-};\ZZ)\). We have \(\int_{G_{\Pi^-}}\omega_0>0\)
because \(C\) has positive symplectic area and \(L\) has zero
symplectic area. Moreover, it is possible to check that \(K\cdot
G_{\Pi^-}<0\) (this is equivalent to saying that \(V_{\Pi^-}\) is
not a P-resolution of \(\frac{1}{11}(1,3)\)).

If we do the same construction for the almost toric picture at the
end of the sequence:
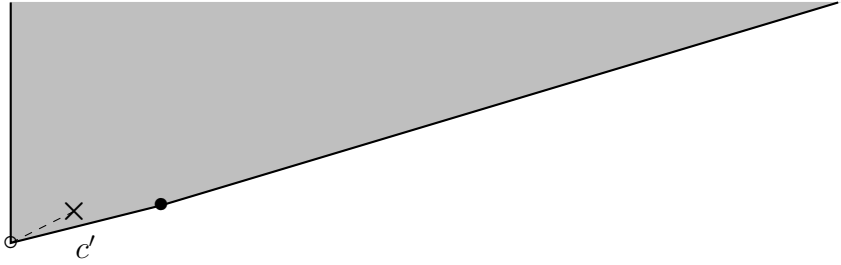
\begin{center}
\begin{tikzpicture}
\filldraw[lightgray] (0,5.5) -- (0,2.3) node[black] {\(\circ\)} -- (2,2.3+0.5) node [black] {\(\bullet\)} -- (11,5.5) -- cycle;
\draw[draw=black,thick] (0,5.5) -- (0,2.3) -- (2,2.8) -- (11,5.5);
\draw[dashed] (0,3/25+1+1.6-21/50) -- (11/25+2-1.6,3/25+1+1.6) node {\(\bm{\times}\)};
\node at (1,2.3+0.25) [below] {\(c'\)};
\end{tikzpicture}
\end{center}
then we get a cycle \(G_{\Pi^+}\) with
\(\int_{G_{\Pi^+}}\omega_1>0\), but now \(K\cdot G_{\Pi^+}>0\).

Since \(H^2(U_{\Pi^-};\RR)=\RR\), the symplectic form falls into one
of the following classes:
\begin{itemize}
\item positively monotone (if its integral over some cycle \(D\) is
positive but \(K\cdot D<0\))
\item negatively monotone (if its integral over some cycle \(D\) is
positive and \(K\cdot D>0\))
\item exact (if its integral over any cycle is zero).
\end{itemize}
Clearly \(\omega_0\) is positively monotone, \(\omega_1\) is
negatively monotone (and, for some \(t\in(0,1)\), \(\omega_t\) is
exact). \qedhere

\end{proof}
\begin{Remark}
If we revisit the proof of the Koll\'{a}r--Shepherd-Barron
classification, we see that there is a magical step where you apply
the semistable minimal model program to the total space of your
smoothing. In the current context, this means that you apply Mori
flips to replace \(K\)-negative curves by \(K\)-positive
curves. Before the flip, in fact, \(-K\) is ample on the central
fibre and hence (by openness of amplitude) relatively ample on some
neighbourhood of the central fibre, so we get an anticanonical
(positively monotone) symplectic form on nearby fibres. After the
flip, \(K\) is relatively ample, and we get a canonical (negatively
monotone) symplectic form on every fibre. Our construction
interpolates between these two symplectic forms. The fact that the
surfaces before and after the flip really coincide with the toric
surfaces \(V_{\Pi^-}\) and \(V_{\Pi^+}\) follows from the paper
\cite{HTU}.

\end{Remark}
\subsection{Infinitely many mutations}

\subsubsection{Nodal trades}

The construction of a Lagrangian torus fibration on the smoothing of
the \(\frac{1}{dp^2}(1,dpq-1)\) singularity works just as well when we
have a smooth point (e.g. \(d=p=1\), \(q=0\)) and gives us the
following almost toric picture of the ball:

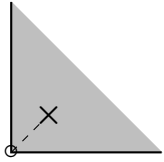
\begin{center}
\begin{tikzpicture}
\filldraw[lightgray] (0,2) -- (0,0) node[black] {\(\circ\)} -- (2,0) -- cycle;
\draw[thick,black] (0,2) -- (0,0) -- (2,0);
\draw[dashed] (0,0) -- (1/2,1/2) node {\(\bm{\times}\)};

\end{tikzpicture}
\end{center}
I like to call this Hamiltonian system the {\em Auroux system} because
I learned of it from Auroux's classic paper on Fano mirror symmetry
\cite{Auroux}.

In a toric moment polygon, a corner corresponds to a smooth point of
the toric variety if it has a neighbourhood which is \(\ZZ\)-affine
isomorphic to the nonnegative quadrant in \(\RR^2\). We can therefore
replace such a local Lagrangian torus fibration with the Auroux
system. Symington calls this operation a {\em nodal trade} and shows
{\cite[Theorem 6.5]{Symington}} that it doesn't change the
symplectomorphism type of the manifold, only the Lagrangian torus
fibration.

\subsubsection{Back to \(\frac{1}{11}(1,3)\)}

Let us return to our example \(U_{\Pi^-}\):

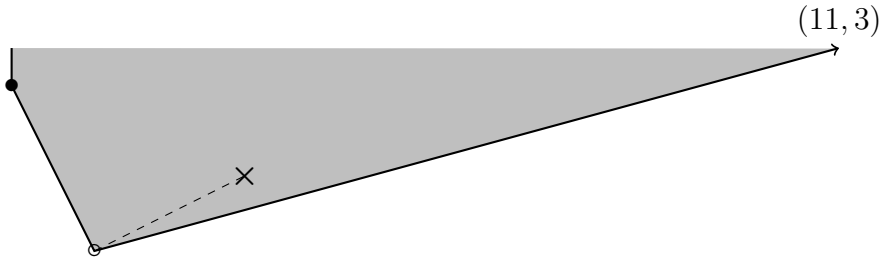
\begin{center}
\begin{tikzpicture}
\filldraw[draw=black,thick,->,fill=lightgray] (0,3) -- (0,2.5) node[black] {\(\bullet\)} -- (2.5*11/25,2.5*3/25) node[black] {\(\circ\)} -- (11,3);
\node at (11,3) [above] {\((11,3)\)};
\draw[dashed] (2.5*11/25,2.5*3/25) -- (2.5*11/25+2,2.5*3/25+1) node {\(\bm{\times}\)};

\end{tikzpicture}
\end{center}
We may perform a nodal trade at the remaining vertex to obtain the
following almost toric picture:

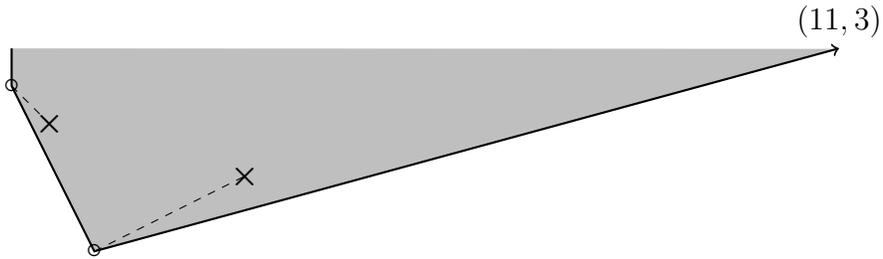
\begin{center}
\begin{tikzpicture}
\filldraw[draw=black,thick,->,fill=lightgray] (0,3) -- (0,2.5) node[black] {\(\circ\)} -- (2.5*11/25,2.5*3/25) node[black] {\(\circ\)} -- (11,3);
\node at (11,3) [above] {\((11,3)\)};
\draw[dashed] (2.5*11/25,2.5*3/25) -- (2.5*11/25+2,2.5*3/25+1) node {\(\bm{\times}\)};
\draw[dashed] (0,2.5) -- (0.5,2) node {\(\bm{\times}\)};

\end{tikzpicture}
\end{center}
Let us perform a nodal slide to make the right-hand branch cut very
small:

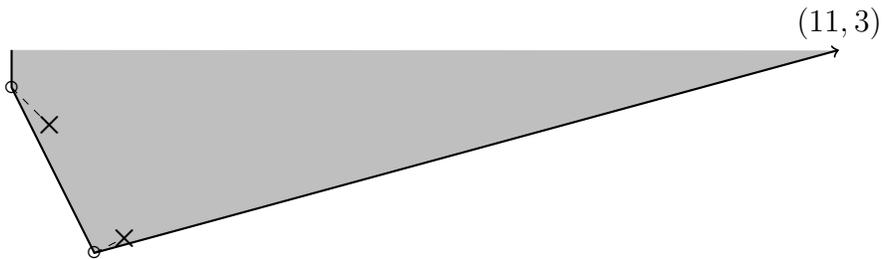
\begin{center}
\begin{tikzpicture}
\filldraw[draw=black,thick,->,fill=lightgray] (0,3) -- (0,2.5) node[black] {\(\circ\)} -- (2.5*11/25,2.5*3/25) node[black] {\(\circ\)} -- (11,3);
\node at (11,3) [above] {\((11,3)\)};
\draw[dashed] (2.5*11/25,2.5*3/25) -- (2.5*11/25+2*0.2,2.5*3/25+1*0.2) node {\(\bm{\times}\)};
\draw[dashed] (0,2.5) -- (0.5,2) node {\(\bm{\times}\)};

\end{tikzpicture}
\end{center}
and mutate by rotating the left-hand branch cut anticlockwise 180
degrees; this means applying a matrix \(M\) to the lower half of the
diagram and we can figure out \(M\) as follows:
\begin{itemize}
\item we know that the vertices marked with circles are not really
breaking points and the edge passing through that point is a
straight line, so we need
\(M\begin{pmatrix}1\\-2\end{pmatrix}=\begin{pmatrix}0\\-1\end{pmatrix}\).
\item we know that the vector \((1,-1)\) is a 1-eigenvector of \(M\).
\end{itemize}
The only possibility is \(M=\begin{pmatrix}2 & 1 \\ -1 &
0\end{pmatrix}\), and we see that the edge pointing in the
\((11,3)\)-direction will end up pointing in the
\((25,-11)\)-direction:

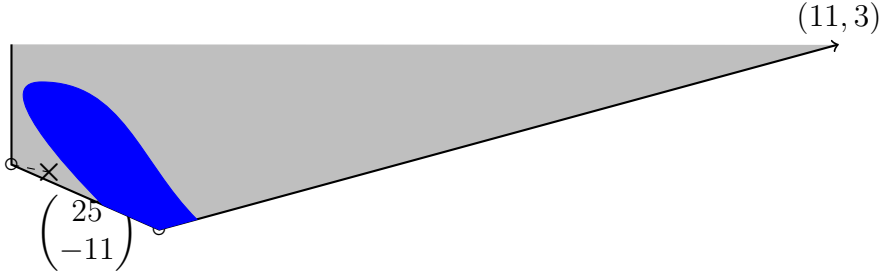
\begin{center}
\begin{tikzpicture}
\filldraw[draw=black,thick,->,fill=lightgray] (0,3) -- (0,7/5) node[black] {\(\circ\)} -- (55/28,15/28) node[black] {\(\circ\)} -- (11,3);
\node at (11,3) [above] {\((11,3)\)};
\draw[dashed] (0,7/5) -- (0.1*5,7/5-0.1) node {\(\bm{\times}\)};
\draw[dashed] (55/28,15/28) -- (0.5,2) node {\(\bm{\times}\)};
\node at (1,0.5) {\(\begin{pmatrix}25\\-11\end{pmatrix}\)};
\filldraw[opacity=0.3,blue] (55/28,15/28) -- (55/28+0.5,15/28+3/22) to[out=135,in=0] (0.4,2.5) to[out=180,in=135] (55/28-0.75,15/28+0.75*11/25) -- cycle;

\end{tikzpicture}
\end{center}
(and the other branch cut, which formerly pointed in the
\((2,1)\)-direction now points in the \((5,-2)\)-direction.

\begin{Lemma}
The preimage of the blue region in this picture is a copy of
\(B_{14,9}\).
\end{Lemma}
\begin{proof}
The matrix \(\begin{pmatrix}11 & 25 \\ 7 & 16\end{pmatrix}\) sends
\((-25,11)\) and \((11,3)\) to \((0,1)\) and \((196,125)\)
respectively, so this vertex is locally modelled on
\(\tilde{\pi}(196,125)\). Note that \(196=14^2\) and \(125=14\times
9-1\) so this is the almost toric model for \(B_{14,9}\). \qedhere

\end{proof}
After making another nodal slide to put this branch cut out of the
way, we can now mutate at the other branch cut and we obtain a copy of
\(B_{37,24}\). Continuing in this manner, we get a whole sequence of
symplectically embedded rational homology balls.

\begin{Remark}
It is a slightly fiddly exercise in affine geometry to check that
the \(p_i,q_i\) we obtain in this manner satisfy the recursion
formula stated in \cref{thm:eu}; a sequence \(p_i,q_i\) satisfying
this recursion is called a {\em Mori sequence}.

\end{Remark}
\begin{Lemma}\label{cor:mori}
If \(X\) is a compact symplectic 4-manifold which admits a
symplectic embedding of \(U_{\Pi^+}\) then some symplectic
deformation of \(X\) admits a symplectic embedding of infinitely
many rational homology balls \(B_{p_i,q_i}\) for a Mori sequence
\((p_i,q_i)\).
\end{Lemma}
\begin{proof}
Let us be slightly more precise about the size of
neighbourhoods. Let \(\Pi^-(a,h)\) denote the bounded polygon
obtained by rescaling \(\Pi^-\) until the compact edge \(c\) (with
slope \(-2\)) has affine length \(a\) and truncating\footnote{Here
we just mean to excise everything above height \(h\) from the
manifold, not to make a symplectic cut at that height.} from above
at height \(h\). We have seen that you can deform the symplectic
form on \(U_{\Pi^+}\) to get \(U_{\Pi^-}(a,h)\) for some \(a,h\). By
a further deformation, we can make \(a\) arbitrarily small (this
amounts to moving the truncating edge \(c\) closer to the origin,
and is known in symplectic geometry as {\em symplectic
deflation}). It is another fiddly exercise in affine geometry
{\cite[Lemma 3.13]{EU}} to show that you can do arbitrarily many
mutations to \(\Pi^-(a,h)\) if \(a\ll h\). Therefore
\(U_{\Pi^-(a,h)}\), the part of \(U_{\Pi^-}\) living over
\(\Pi^-(a,h)\)) contains a Mori sequence of rational homology balls,
which we now find in this deformation of \(X\). \qedhere

\end{proof}
We now apply this result to the quintic surface.

\subsection{The quintic surface}

Consider the surface \(\cp{1}\times\cp{1}\) and let \(B\) be a curve
of bidegree \((6,6)\) which:
\begin{itemize}
\item intersects the diagonal \(\cp{1}\) at six points each with
multiplicity 2,
\item intersects some fixed ruling \(\cp{1}\times\{z\}\) at three points
each with multiplicity 2.
\end{itemize}
Take a double cover of \(\cp{1}\times\cp{1}\) branched along
\(B\). The preimage of the diagonal is then a pair of rational
\(-4\)-curves (intersecting at six points) and the preimage of the
ruling is a pair of rational \(-3\)-curves (intersecting at three
points). If we pick one of these \(-4\)-curves \(C_1\) and one of
these \(-3\)-curves \(C_2\) then they intersect at a single point
(just like the diagonal and the ruling).

Let us collapse the curves \(C_1,C_2\). We obtain a singularity of
type \(\frac{1}{11}(1,3)\): collapsing a chain of rational curves
\(C_1,\ldots,C_\ell\) with self-intersections \(-b_1,\ldots,-b_\ell\)
like this yields the cyclic quotient singularity
\(b_1-\frac{1}{b_2-\frac{1}{\cdots-\frac{1}{b_\ell}}}\), and
\(\frac{11}{3}=4-\frac{1}{3}\).

This singular surface \(Z\) has a P-resolution \(Y\to Z\); \(Y\) is
obtained from the branched double cover by collapsing \(C_1\). This is
the P-resolution we have been studying (corresponding to the polygon
\(\Pi^+\)). The surface \(Y\) is a stable quintic surface with a
\(\frac{1}{4}(1,1)\)-singularity. By a theorem of Rana {\cite[Theorem
4.10]{Rana}}, this quintic is smoothable and, by construction, this
smoothing contains a symplectically embedded \(U_{\Pi^+}\) (symplectic
with respect to the (negatively monotone) canonical symplectic
form). By \cref{cor:mori}, we can make a deformation of the symplectic
form on the quintic, supported on \(U_{\Pi^+}\), so that the new
symplectic form admits symplectic embeddings of \(B_{p_i,q_i}\) for a
Mori sequence \(p_i,q_i\), which proves \cref{thm:eu}.

\appendix
\section{Appendix: P-resolutions}
\label{app:pres}

In this appendix, I explain how you can tell if a partial resolution
of a cyclic quotient singularity is a P-resolution or not. The
appendix is a modification of a blog post I wrote on 4th June 2018
(http://jde27.uk/blog/cyclic2.html).

\subsection{Discrepancies}

Let \(X\) be a cyclic quotient singularity of type
\(\frac{1}{n}(1,a)\) and let \(\pi\colon\tilde{X}\to X\) be its
minimal resolution. The exceptional locus of \(\pi\) comprises a chain
of rational curves \(C_1,\ldots,C_\ell\) with \(C_i^2=-b_i\) where
\[\frac{n}{a}=b_1-\frac{1}{b_2-\frac{1}{\cdots-\frac{1}{b_\ell}}}.\]
(See my earlier blog post http://jde27.uk/blog/cyclic1.html for an
explanation of this in terms of truncations of the moment
polygon). Our first goal is to express the canonical class
\(K_{\tilde{X}}=-c_1(\tilde{X})\in H^2(\tilde{X};\ZZ)\) in terms of
Poincar\'{e}-duals of the closed curves \(C_1,\ldots,C_\ell\) in
\(\tilde{X}\). However, \(\tilde{X}\) is not compact, so we must make
do with Alexander-Lefschetz duality \(H^2(\tilde{X};\ZZ)\cong
H_2(\tilde{X},\partial\tilde{X};\ZZ)\). The compact curves
\(C_1,\ldots,C_\ell\) do not generate
\(H_2(\tilde{X},\partial\tilde{X};\ZZ)\), rather they span the image
of the map \(H_2(\tilde{X};\ZZ)\to
H_2(\tilde{X},\partial\tilde{X};\ZZ)\), whose cokernel is the finite
group \(H_1(\partial\tilde{X};\ZZ)=\ZZ/n\) by the long exact sequence
of the pair \((\tilde{X},\partial\tilde{X})\). If we work over \(\QQ\)
then this cokernel vanishes, but we are then forced to write
\(K_{\tilde{X}}\) as a {\em rational} linear combination of the
Alexander-Lefschetz duals of \(C_1,\ldots,C_\ell\):
\[K_{\tilde{X}}=\sum k_iC_i.\] The numbers \(k_i\) are called the {\em
discrepancies} of the singularity (they measure the discrepancy
between \(K_{\tilde{X}}\) and \(\pi^*K_X=0\)).

We have the adjunction formula \(K_{\tilde{X}}\cdot C_i=-C_i^2-2\) for
each \(i\). Substituting \(K_{\tilde{X}}=\sum k_iC_i\), we obtain a
system of simultaneous equations \[\sum k_iE_i\cdot E_j=b_j-2,\quad
j=1,\ldots,\ell,\] for the discrepancies.

\begin{Example}
Consider the minimal resolution of the \(\frac{1}{4}(1,1)\)
singularity which has exceptional locus \(C_1\) with
\(C_1^2=-4\). Then we have \(-4k_1=2\) so \(k_1=-1/2\).

\end{Example}
\begin{Example}
Consider the minimal resolution of the \(\frac{1}{11}(1,3)\)
singularity which has exceptional locus \(C_1\cup C_2\) with
\(C_1^2=-4\), \(C_2^2=-3\) and \(C_1\cdot C_2=1\). Then we have
\[-4k_1+k_2=2,\qquad k_1-3k_2=1,\] or \(k_1=-7/11\), \(k_2=-6/11\).

\end{Example}
\begin{Remark}
A singularity is called {\em terminal} if all the discrepancies are
positive, {\em canonical} if they are nonnegative, {\em log
terminal} if they are \(>-1\) and {\em log canonical} if they are
\(\geq -1\). We can see that these singularities are all log
terminal but not canonical (indeed, canonical surface singularities
are precisely the ADE singularities; there are no nonsmooth terminal
surface singularities).

\end{Remark}
\subsection{Is it a P-resolution or not?}

Recall that a partial resolution \(g\colon Z\to\CC^2/G\) of a cyclic
quotient singularity is called a P-resolution if it has at worst
T-singularities and if \(K_Z\cdot D>0\) for any curve \(D\subset Z\)
contracted by \(g\). Let \(f\colon Y\to Z\) be a resolution of
singularities of \(Z\). Let \(\cup_iE_i\) be the exceptional locus of
\(f\) and let \(k_i\) be the discrepancies (so \(K_Y=f^*K_Z+\sum
k_iE_i\)). If \(D\subset Z\) is an irreducible curve then there is a
unique irreducible curve \(C\subset Y\) such that \(f_*C=D\). We have
\[K_Z\cdot D=f^*K_Z\cdot C=(K_Y-\sum k_iE_i)\cdot C.\]

\begin{Example}
Suppose that \(Y\) is the minimal resolution of
\(\frac{1}{11}(1,3)\) (which contains \(C_1,C_2\) with \(C_1^2=-4\),
\(C_2^2=-3\)) and \(Z\) is obtained from \(Y\) by collapsing
\(C_1\). We have
\[K_Z\cdot f_*C_2=K_Y\cdot C_2-(-1/2)C_1\cdot C_2.\]
Since \(K_Y\cdot C_2=-C_2^2-2=1\) and \(C_1\cdot C_2=1\), this gives
\[K_Z\cdot f_*C_2=\frac{3}{2}>0,\]
which shows that \(f\colon Y\to Z\) is a P-resolution.

\end{Example}
\begin{Exercise}
Consider the chain \(C_1,C_2\) with \(C_1^2=-4\), \(C_2^2=-3\)
again. Blow up a point on \(C_1\) and a point on the exceptional
curve of this blow-up to obtain a surface \(Y'\) with a chain of
spheres \(E_1,E_2,E_3,E_4\) with \(E_1^2=-1\), \(E_2^2=-2\),
\(E_3^2=-5\), \(E_4^2=-3\). Collapse \(E_2,E_3,E_4\) to obtain a
surface \(Z'\) with a \(\frac{1}{25}(1,14)\) singularity (because
\(2-\frac{1}{5-\frac{1}{3}}=\frac{25}{14}\)). Show that
\(K_{Z'}\cdot (f')_*E_1\) is negative, so that this is not a
P-resolution of \(\frac{1}{11}(1,3)\). [I get \(K_{Z'}\cdot
(f')_*E_1=-3/5\).]
\end{Exercise}
\subsection{Enumerating P-resolutions}

Koll\'{a}r and Shepherd-Barron show {\cite[Lemma 3.13]{KSB}} that any
P-resolution of \(\frac{1}{n}(1,a)\) is obtained from a certain
``maximal resolution'' by contracting some curves to get
T-singularities (equivalently, the smoothing is obtained by rationally
blowing down some chains of curves in the maximal resolution). Since
there is a finite number of curves to contract/rationally blow-down,
this shows there are only finitely many P-resolutions.

Rather than explaining their proof, I will just describe how to find
the maximal resolution. Let \(X\) be the cyclic quotient
singularity. Let's define a poset whose elements are resolutions
\(f\colon Y\to X\) obtained by blowing up intersections between
exceptional curves of the minimal resolution and where \(Y_1<Y_2\) if
the resolution \(f_2\colon Y_2\to X\) factors through a morphism
\(Y_2\to Y_1\). For any \(Y\) in our poset, we have
\(K_{Y}=\sum(\alpha_j(Y)-1)E_j\) for some \(\alpha_j(Y)\in\QQ\). If
\(Y_1<Y_2\) then \(\max_j\alpha_j(Y_1)>\max_j\alpha_j(Y_2)\) (i.e. as
you move up in the poset, the number \(\max_j\alpha_j(Y)\)
increases). The maximal resolution is defined to be the maximal
element in the poset for which \(\max_j\alpha_j(Y)<1\).

\begin{Example}
Consider the case \(n=11,a=3\). The minimal resolution is a chain of
spheres \(C_1,C_2\) with self-intersections \(-4,-3\). Blow up the
unique intersection point to obtain a \(-1\)-sphere \(E\) (and
continue to write \(C_1,C_2\) for the proper transforms of
\(C_1,C_2\)). We can compute that
\[K=-\frac{7}{11}C_1-\frac{2}{11}E-\frac{6}{11}C_2,\] giving
\(\alpha_1=4/11\), \(\alpha_2=9/11\), \(\alpha_3=5/11\). If we
blow-up the point \(C_1\cap E\) then we get \(\alpha_2=13/11>1\)
(for the new exceptional sphere). If we blow-up the point \(C_2\cap
E\) then we get \(\alpha_3=14/11>1\) (for the new exceptional
sphere). So the maximal resolution is the one-point blow-up of the
minimal resolution. This has a chain of three spheres with
self-intersections \(-5,-1,-4\). We can collapse the \(-4\)-curve to
get a surface with a T-singularity (this is a P-resolution of
\(\frac{1}{11}(1,3)\)). We can collapse the \(-1\) curve and then
the \(-4\) curve to get a surface with a T-singularity (this is
another P-resolution). This exhausts the possible P-resolutions.

\end{Example}
\begin{Exercise}
Work through the example of \(\frac{1}{19}(1,7)\) ({\cite[Example
3.15]{KSB}}) which I discussed in \cref{exm:KSBpres}.

\end{Exercise}
\bibliography{kias}

\begin{thebibliography}{10}

\bibitem{Auroux}
D.~Auroux.
\newblock Mirror symmetry and {T}-duality in the complement of an anticanonical
  divisor.
\newblock {\em Journal of G\"okova Geometry Topology}, 1:51--91, 2007.

\bibitem{BhupalOno}
M.~Bhupal and K.~Ono.
\newblock Symplectic fillings of links of quotient surface singularities.
\newblock {\em Nagoya Math. J.}, 207:1--45, 2012.

\bibitem{CieliebakEliashberg}
K.~Cieliebak and Ya. Eliashberg.
\newblock {\em From {S}tein to {W}einstein and back}, volume~59 of {\em
  American Mathematical Society Colloquium Publications}.
\newblock American Mathematical Society, Providence, RI, 2012.
\newblock Symplectic geometry of affine complex manifolds.

\bibitem{Delzant}
T.~Delzant.
\newblock Hamiltoniens p\'{e}riodiques et images convexes de l'application
  moment.
\newblock {\em Bull. Soc. Math. France}, 116(3):315--339, 1988.

\bibitem{Duistermaat}
J.~J. Duistermaat.
\newblock On global action-angle coordinates.
\newblock {\em Comm. Pure Appl. Math.}, 33(6):687--706, 1980.

\bibitem{EvansBook}
J.~D. Evans.
\newblock {\em Lectures on {L}agrangian torus fibrations}, volume 105 of {\em
  London Mathematical Society Student Texts}.
\newblock Cambridge University Press, Cambridge, 2023.

\bibitem{ES1}
J.~D. Evans and I.~Smith.
\newblock Markov numbers and {L}agrangian cell complexes in the complex
  projective plane.
\newblock {\em Geom. Topol.}, 22(2):1143--1180, 2018.

\bibitem{ES2}
J.~D. Evans and I.~Smith.
\newblock Bounds on {W}ahl singularities from symplectic topology.
\newblock {\em Algebr. Geom.}, 7(1):59--85, 2020.

\bibitem{EU}
J.~D. Evans and G.~Urz\'{u}a.
\newblock Antiflips, mutations, and unbounded symplectic embeddings of rational
  homology balls.
\newblock {\em Ann. Inst. Fourier (Grenoble)}, 71(5):1807--1843, 2021.

\bibitem{GalkinMikhalkin}
S.~Galkin and G.~Mikhalkin.
\newblock Singular symplectic spaces and holomorphic membranes.
\newblock {\em Eur. J. Math.}, 8(3):932--951, 2022.

\bibitem{HP1}
P.~Hacking and Yu. Prokhorov.
\newblock Degenerations of {D}el {P}ezzo surfaces {I}.
\newblock {\em arXiv:0509529}, 2005.

\bibitem{HP2}
P.~Hacking and Yu. Prokhorov.
\newblock Smoothable del {P}ezzo surfaces with quotient singularities.
\newblock {\em Compositio Mathematica}, 146(1):169--192, 2010.

\bibitem{HTU}
P.~Hacking, J.~Tevelev, and G.~Urz\'ua.
\newblock Flipping surfaces.
\newblock {\em J. Algebraic Geom.}, 26(2):279--345, 2017.

\bibitem{KSB}
J.~Koll\'ar and N.~I. Shepherd-Barron.
\newblock Threefolds and deformations of surface singularities.
\newblock {\em Invent. Math.}, 91(2):299--338, 1988.

\bibitem{Lisca}
P.~Lisca.
\newblock On symplectic fillings of lens spaces.
\newblock {\em Transactions of the American Mathematical Society},
  360:765--799, 2008.

\bibitem{LooijengaWahl}
E.~Looijenga and J.~Wahl.
\newblock Quadratic functions and smoothing surface singularities.
\newblock {\em Topology}, 25(3):261--291, 1986.

\bibitem{McDuffSalamon}
D.~McDuff and D.~Salamon.
\newblock {\em Introduction to symplectic topology}.
\newblock Oxford University Press, second edition, 2005.

\bibitem{NPP}
A.~N\'emethi and P.~Popescu-Pampu.
\newblock On the {M}ilnor fibers of cyclic quotient singularities.
\newblock {\em Proceedings of the London Mathematical Society},
  101(3):554--588, 2010.

\bibitem{Rana}
J.~Rana.
\newblock A boundary divisor in the moduli spaces of stable quintic surfaces.
\newblock {\em Internat. J. Math.}, 28(4):1750021, 61, 2017.

\bibitem{RU}
J.~Rana and G.~Urz\'ua.
\newblock Bounding {T}-singularities on {KSBA} surfaces.
\newblock {\em Advances in Mathematics}, 345:814--844, 2019.

\bibitem{Symington}
M.~Symington.
\newblock Four dimensions from two in symplectic topology.
\newblock In {\em Topology and geometry of manifolds ({A}thens, {GA}, 2001)},
  volume~71 of {\em Proc. Sympos. Pure Math.}, pages 153--208. Amer. Math.
  Soc., Providence, RI, 2003.

\bibitem{Zung}
N.~T. Zung.
\newblock Symplectic topology of integrable {H}amiltonian systems, {II}:
  topological classification.
\newblock {\em Compositio Mathematica}, 138(2):125--156, 2003.

\end{thebibliography}
\bibliographystyle{plain}
\end{document}